\documentclass[a4paper,11pt]{amsart}
\usepackage{amsmath,amssymb,latexsym,amsthm,newlfont,enumerate,graphicx, longtable}
\usepackage[all]{xy}
\usepackage{url, pst-plot, pstricks, pstricks-add}
\makeindex
\renewcommand{\tilde}{\widetilde}

\theoremstyle{plain}

\newtheorem{thm}{Theorem}[section]
\newtheorem{pro}[thm]{Proposition}
\newtheorem{con}[thm]{Conjecture}
\newtheorem{lem}[thm]{Lemma}
\newtheorem{case}{Case}
\theoremstyle{definition}

\newtheorem{dfn}[thm]{Definition}
\newtheorem{nt}[thm]{Notation}
\newtheorem{rem}[thm]{Remark}
\newtheorem{exa}[thm]{Example}
\DeclareMathOperator{\W}{WDiv}
\DeclareMathOperator{\Div}{Div}
\DeclareMathOperator{\Pic}{Pic} 
\DeclareMathOperator{\Nef}{Nef}
\DeclareMathOperator{\Eff}{\mathrm{Eff}}
\DeclareMathOperator{\Effb}{\overline{\mathrm{Eff}}}
\DeclareMathOperator{\Movb}{\overline{\mathrm{Mov}}}
\DeclareMathOperator{\Supp}{Supp}
\DeclareMathOperator{\inte}{int}
\DeclareMathOperator{\Proj}{Proj}
\DeclareMathOperator{\B}{Big}
\DeclareMathOperator{\Spec}{Spec}
\DeclareMathOperator{\Aut}{Aut}
\DeclareMathOperator{\Sing}{Sing}
\DeclareMathOperator{\Exc}{Exc}
\DeclareMathOperator{\TV}{TV}

\newcommand{\OO}{\mathcal{O}}
\DeclareMathOperator{\mult}{mult}

\renewcommand{\tilde}{\widetilde}

\renewenvironment{itemize}{
  \begin{list}{}{
    \setlength{\leftmargin}{1em}
    \setlength{\itemsep}{0.1em}
   \setlength{\parskip}{0pt}
    \setlength{\parsep}{0.15em}
  }
}{
  \end{list}
}

\begin{document}

\title{Non-rigid quartic $3$-folds}
\author{Hamid Ahmadinezhad}
\address{School of Mathematics, University of Bristol, Bristol, BS8 1TW, UK}
\email{h.ahmadinezhad@bristol.ac.uk}
\author{Anne-Sophie Kaloghiros}
\address{Imperial College London, 180 
Queen's Gate, London SW7 2AZ, UK}
\curraddr{Brunel University London, Uxbridge UB8 3PH}
\email{anne-sophie.kaloghiros@brunel.ac.uk}
\thanks{The first author was partially supported by the Austrian Science Fund (FWF) [grant number P22766-N18] and
the second author was supported by the Engineering and Physical Sciences Research Council [grant number EP/H028811/1].}\begin{abstract}
Let $X\subset \mathbb{P}^4$ be a terminal factorial quartic $3$-fold. If $X$ is non-singular, $X$ is \emph{birationally rigid}, i.e.~the classical MMP on any terminal $\mathbb{Q}$-factorial projective variety $Z$ birational to $X$ always terminates with $X$. This no longer holds when $X$ is singular, but very few examples of non-rigid factorial quartics are known. In this article, we first bound the local analytic type of singularities that may occur on a terminal factorial quartic hypersurface $X\subset \mathbb{P}^4$. A singular point on such a hypersurface is either of type $cA_n$ ($n\geq 1$), or of type $cD_m$ ($m\geq 4$), or of type $cE_6, cE_7$ or $cE_8$. We first show that if $(P \in X)$ is of type $cA_n$, $n$ is at most $7$, and if $(P \in X)$ is of type $cD_m$, $m$ is at most $8$. We then construct examples of non-rigid factorial quartic hypersurfaces whose singular loci consist (a) of a single point of type $cA_n$ for $2\leq n\leq 7$ (b) of a single point of type $cD_m$ for $m= 4$ or $5$ and (c) of a single point of type $cE_k$ for $k=6,7$ or $8$.  
 \end{abstract}
\maketitle

\section{Introduction}

A classical problem in algebraic geometry is to determine which quartic hypersurfaces in $\mathbb{P}^4$ are rational. In their seminal paper \cite{IM}, Iskovskikh and Manin prove that a nonsingular quartic hypersurface $X_4\subset \mathbb{P}^4$ is {\it birationally rigid} (see the precise definition below) and, in particular, is not rational.

The classical Minimal Model Program (MMP) shows that a uniruled projective $3$-fold $Z$ with terminal singularities is birational to a {\it Mori fibre space} $X/S$. More precisely, there is a small morphism $f\colon \tilde Z\to Z$ where $\tilde Z$ is terminal and $\mathbb{Q}$-factorial (see \cite[Section 6.3]{KM98}) and the classical MMP $\psi \colon \tilde Z \dashrightarrow X$  terminates with a Mori fibre space $X/S$ (see \cite[Section 3]{KM98}). Neither the morphism $f$ nor the birational map $\psi$ is unique in general. Mori fibre spaces are end products of the MMP and hence should be seen as {\it distinguished representatives} of their classes of birational equivalence. In general, there may be more than one Mori fibre space in a class of birational equivalence. The {\it pliability} of a uniruled terminal $3$-fold $Z$ is the set of distinguished representatives in its class of birational equivalence, that is:
\[\mathcal P(Z) =\{X/T \mbox{ Mori fibre space birational to } Z \}/ \sim,\]
 where $\sim$ denotes the {\it square birational equivalence} defined in \cite[Definition 5.2]{Co95}.  If $X$ itself is a Mori fibre space, $X$ is called {\it birationally rigid} if its pliability is $\mathcal P(X)=\{[X]\}$.

A quartic hypersurface $X\subset \mathbb{P}^4$ with terminal singularities is a Mori fibre space precisely when it is factorial, that is when every Weil divisor on $X$ is Cartier. The two conditions on the singularities of a Mori fibre space are quite different: requiring that the singularities of $X$ are terminal is a local analytic condition, while factoriality is a global topological condition.  
Mella extended Iskovskikh and Manin's result and proved that terminal factorial quartic hypersurfaces with no worse than ordinary double points are birationally rigid (see \cite[Theorem 2]{Me04}). 
Factoriality is a crucial condition for this to hold. Indeed, general determinantal quartic hypersurfaces are examples of rational nodal quartic hypersurfaces (see the introductions of \cite{Pettersen, Me04}) but are not factorial. Todd discusses several examples of non-factorial rational nodal quartic hypersurfaces: the Burkhardt quartic studied in \cite{To36} has $45$ nodes (see also \cite[Section 5.1]{Pettersen}); an example with $36$ nodes is mentioned in \cite{T33} (see also \cite[Example 6.4.2]{Pettersen}) and two examples with $40$ nodes are studied in \cite{T35} (see also \cite[Examples 6.2.1, 6.2.2]{Pettersen}). 
 In fact, {\it most} terminal non-factorial quartic hypersurfaces are rational \cite{Kal09}.

However, factoriality alone is not sufficient to guarantee birational rigidity. There are several known examples of non-rigid terminal factorial quartic hypersurfaces: an example with a $cA_2$ point is studied in \cite{CM04} and entry No.~5 in Table 1 of \cite{Ah12} is an example with a $cA_3$ point. In this paper, we show that these examples are not pathological by constructing many examples of non-rigid terminal factorial quartic $3$-folds with a singular point of type $cA_n$ for $n\geq 2$.   
It is conjectured that a terminal factorial quartic $3$-fold with no worse than $cA_1$ points is rigid; we address this conjecture in forthcoming work. 

Terminal $3$-fold hypersurfaces have isolated cDV singularities \cite[(3.1)]{YPG}; the local analytic type of a singular point thus belongs to one of two infinite families-- $cA_n$ for $n\geq 1$, or $cD_m$ for $m \geq 4$-- or is $cE_6,cE_7$ or $cE_8$. 
The first step in our study is to bound the local analytic type of singularities that can occur on a terminal factorial quartic $3$-fold. We use topology and singularity theory to bound the local analytic type in the $cA_n$ and $cD_m$ case, and we show:
\begin{pro} \label{thepro}
If $(P\in X)$ is a $cA_n$ (resp.~ $cD_m$) point on a terminal factorial quartic hypersurface $X\subset \mathbb{P}^4$, then $n\leq 7$ (resp.~$m\leq 8$).
\end{pro}
The methods used to prove Proposition~\ref{thepro} do not restrict the local analytic type of points of type $cE$. In fact, all possible local analytic types of $cE$ points are realised: we give examples of terminal factorial quartic hypersurfaces with isolated singular points of type $cE_6, cE_7$ or $cE_8$ (see Example~\ref{cE6}).
As is noted in Remark~\ref{sharp}, the bound on the local analytic type of $cA$ points is sharp, but we do not believe that the bound is optimal in the $cD$ case .

If $X$ is a terminal factorial quartic $3$-fold, the Sarkisov Program shows that any birational map $X\dashrightarrow X^\prime$ to a Mori fibre space $X^\prime/S^\prime$ is the composition of finitely many Sarkisov links (see Section~\ref{general} for definitions and precise statements). Thus $X$ is non-rigid precisely when there exists a link $X\dashrightarrow X^\prime$ where $X^\prime/S^\prime$ is a Mori fibre space. Such a link is initiated by a morphism $f\colon Z \to X$, where $Z$ is terminal and $\mathbb{Q}$-factorial, and $f$ contracts a divisor to a singular point or to a curve passing through a singular point. In general, little is known about the explicit form of the morphism $f$. When $f$ contracts a divisor to a $cA_n$ point $(P\in X)$, Kawakita shows that the germ of $f$ is a  weighted blowup, and classifies possible weights according to the local analytic type of $(P \in X)$ \cite{Kawk01,Kawk02, Kawk03}.

For each $n$ with $2\leq n\leq 7$, we write down the equation of a quartic hypersurface $X$ with a morphism $f\colon Z \to X$ that contracts a divisor to a $cA_n$ point and initiates a Sarkisov link. After a suitable embedding of $X$ as a complete intersection in a larger weighted projective space $\mathbb{P}=\mathbb{P}(1^5, \alpha, \beta)$, we recover $f$ as the restriction of a weighted blowup $\mathcal F \to \mathbb{P}$ whose weights are determined by Kawakita's classification.
The variety $\mathcal F$ is a toric variety of Picard rank $2$, and therefore it is possible to write down explicitly all contracting rational maps $\mathcal F \dashrightarrow U$ to a projective variety $U$. We then check that the birational geometry of $\mathcal F$ induces a Sarkisov link $X\dashrightarrow X^\prime$, where $X^\prime/S^\prime$ is a Mori fibre space. 

To our knowledge, our construction is the first use of Kawakita's classification to write down explicit {\sl global} uniruled $3$-fold extractions. We prove:

\begin{thm} If $(P\in X)$ is a singular point of type $cA_n$ on a terminal factorial quartic $3$-fold, then $n\leq 7$.  There are examples of non-rigid terminal factorial quartic $3$-fold with a singular point of type $cA_n$ for $2\leq n\leq 7$. \end{thm}   

We also give examples of non-rigid terminal factorial quartic $3$-folds with $cD_4$, $cD_5$ and $cE_6, cE_7$ and $cE_8$ singular points (Examples~\ref{cD4}, \ref{cD5}, and \ref{cE6}). We make the following general conjecture, which generalises \cite[Section 1.3 and Theorem 1.6]{CM04}. 
\begin{con} 
Let $X\subset \mathbb{P}^4$ be a terminal factorial quartic hypersurface. Then $\mathcal P(X)$ is finite and $\mathcal P(X)=\{[X]\}$ precisely when $X$ has no worse than $cA_1$ singularities. 
In particular, no terminal factorial quartic hypersurface is rational. 
\end{con}

\subsection*{Outline of the paper}
Section~\ref{general} recalls general results on the Sarkisov program-- that is, on the study of birational maps between Mori fibre spaces-- in dimension $3$ and on the geography of models of Mori dream spaces. 
When $X$ is a terminal $\mathbb{Q}$-factorial Fano $3$-fold with $\rho(X)=1$, a Sarkisov link $X \dashrightarrow X^\prime$ is initiated by a morphism $f\colon Z\to X$ that contracts a single divisor. Here, we state precise conditions for $f$ to initiate a Sarkisov link. We urge the reader who is mainly interested in explicit examples and in bounds on singularities to skip this section on a first reading and refer back to it as and when needed.  

Section~\ref{terminal} collects results on terminal singularities in dimension $3$. We concentrate on the case of terminal Gorenstein singularities,  which are those that appear on hypersurfaces. We use the existence of smoothings of terminal Gorenstein Fano $3$-folds to bound the local analytic type of singularities on a terminal quartic hypersurface. Last, we recall Kawakita's classification of the germs of divisorial contractions with centre at a $cA_n$ point in terms of the local analytic type of that point. 

Section~\ref{examples} presents our examples of non-rigid terminal factorial quartic $3$-folds. We consider hypersurfaces $X\subset \mathbb{P}^4$ that can be embedded as general complete intersections of type $(2,2,4)$ in a weighted projective space $\mathbb{P}= \mathbb{P}(1^5, 2^2)$. For suitable weighted blowups $F\colon \mathcal F \to\mathbb{P}$, the restriction $f=F_{|Z}\colon Z\to X$ (where $Z$ is the proper transform of $X$) is a divisorial contraction with centre at a $cA_n$ point, and the birational geometry of $\mathcal F$ induces a Sarkisov link $X\dashrightarrow X^\prime$. We give examples of non-rigid quartic hypersurfaces with a $cA_n$ point for all $2\leq n \leq 7$, and explain our construction in detail in a few cases. We also give examples of non-rigid terminal factorial quartic hypersurfaces with singular points of type $cD$ and $cE$.         

{\bf Acknowledgements.} We would like to thank Ivan Cheltsov, Alessio Corti, Masayuki Kawakita, Yuri Prokhorov, Josef Schicho and Constantin Shramov for useful comments and conversations. We thank the referee for many helpful suggestions.

\section{Preliminary results}\label{general}
Throughout this paper, we work with normal projective varieties over $\mathbb C$. 
\subsection*{General results in birational geometry.} Let $X$ be a normal projective variety and $\mathbf k\in \{\mathbb Z,\mathbb Q,\mathbb R\}$. We denote by $\W_{\mathbf k} (X)$ the group of Weil $\mathbf k$-divisors, by $\Div_{\mathbf k}(X)$ the group of $\mathbf k$-Cartier $\mathbf k$-divisors on $X$, and by $\sim_{\mathbf k}$ and $\equiv$ the $\mathbf k$-linear and numerical equivalence of $\mathbb R$-divisors. 
We write $\Pic(X)_\mathbf k=\Div_\mathbf k(X)/\sim_\mathbf k$ and $N^1(X)_\mathbf k=\Div_\mathbf k(X)/\equiv$. 

The 
nef, effective and pseudo-effective cones in $N^1(X)_{\mathbb R}$ are denoted by 
$\Nef (X)$, 
$\Eff(X)$ and $\overline{\Eff}(X)$. The movable cone $\Movb (X)$ is the closure of the cone in $N^1(X)_{\mathbb R}$ spanned by the numerical classes of divisors whose stable base locus has codimension at least $2$. If $\mathcal C\subset N^1(X)_{\mathbb R}$ is a cone, we always denote its closure by $\overline{\mathcal C}$. 

We recall the definitions of models of divisors introduced in \cite{BCHM}.
\begin{dfn} \label{neg}
Let $Z$ be a normal projective variety and $D\in \Div_{\mathbb{Q}}(Z)$.
\begin{enumerate}[1.]
 \item A birational map $f\colon Z \dashrightarrow X$ is \emph{contracting} if $f$ is proper and $f^{-1}$ contracts no divisor. The map $f$ is \emph{small} if both $f$ and $f^{-1}$ are contracting birational maps. 
\item Let $D\in \Div_\mathbb{Q}(Z)$ be a $\mathbb{Q}$-Cartier divisor and let $f\colon Z\dashrightarrow X$ be a contracting birational map such that $f_*D$ is $\mathbb{Q}$-Cartier. The map $f$ is \emph{$D$-nonpositive} if for a resolution $(p,q)\colon W \rightarrow Z\times X$, 
$$p^*D = q^*D^\prime+ E,$$
where $E$ is effective and $q$-exceptional. When $\Supp E$ contains the strict transform of all $f$-exceptional divisors, $f$ is said to be \emph{$D$-negative}.  
\item Assume that $D$ is effective. The map $f$ is a \emph{semiample model} of $D$ if $f$ is $D$-nonpositive, $X$ is normal and projective and $D^\prime$ is semiample. 
If $\varphi \colon X\to S$ is the semiample fibration defined by $D^\prime$, the \emph{ample model} of $D$ is the composition $\varphi\circ f\colon Z \dashrightarrow X\to S$.
\end{enumerate}
\end{dfn}
\begin{nt} If $D=K_Z$ is a canonical divisor on $Z$, we say that a birational contraction is $K$-nonpositive (resp~$K$-negative) instead of $K_Z$-nonpositive (resp.~$K_Z$-negative). 
\end{nt}
\begin{dfn} 
\begin{enumerate}[1.]\item A birational contraction $X \stackrel{\varphi}\dashrightarrow X^\prime$ between $\mathbb{Q}$-factorial varieties is \emph{elementary} if $\varphi$ is either a morphism whose exceptional locus is a prime divisor on X or a small birational map that fits into a diagram
\[ \xymatrix{X \ar[dr]_f \ar@{-->}[rr]^{\varphi} && X^\prime\ar[dl]^{f^\prime}\\&W&}\]
where $f$ and $f^\prime$ are morphisms and the Picard ranks of $X,X^\prime$ and $W$ satisfy
\[ \rho(X)= \rho(X^\prime)= \rho(W)+1.\]
\item
Assume that $X$ is $\mathbb{Q}$-factorial and let $D$ be an effective $\mathbb{Q}$-divisor on $X$. A \emph{$D$-MMP on $X$} is a composition of $D$-nonpositive elementary contractions
$ X \dashrightarrow X_1 \dashrightarrow \cdots \dashrightarrow X_n=X_D$,
where $X_D$ is a semiample model for $D$.
\end{enumerate}
\end{dfn}
\subsection*{Geography of models of Mori dream spaces}
We first recall the definition of Mori dream spaces and the properties that will be important in this paper. 
\begin{dfn}\cite{HK00} Let $Z$ be a projective $\mathbb{Q}$-factorial variety with $\Pic(Z)_\mathbb{Q}= N^1_\mathbb{Q}(Z)$; $Z$ is a Mori dream space if 
\begin{itemize}
\item[(i)] $\Nef(Z)$ is the affine hull of finitely many semiample line bundles,
\item[(ii)] there are finitely many small birational maps $f_i \colon Z \dashrightarrow Z_i$ to projective $\mathbb{Q}$-factorial varieties $Z_i$ satisfying (i) such that $\Movb Z= \cup f_i^*(\mathbb Nef Z_i)$. 
\end{itemize}
\end{dfn}

When $Z$ is a Mori dream space, we may run the $D$-MMP for every $\mathbb{Q}$-divisor $D$. More precisely, there is a finite decomposition \cite[Section 5]{KKL12}:
\[ \Effb Z= \bigsqcup_{i=1}^N \mathcal C_i \mbox{, where for all }i\] 
\begin{enumerate}[(a)]\item $\overline{\mathcal C_i}$ is a rational polyhedral cone,
\item there is a birational contraction to a $\mathbb{Q}$-factorial normal projective variety $\varphi_i \colon Z \dashrightarrow Z_i$ that is the ample model of all $D\in \mathcal C_i$ and a semiample model of all $D\in \overline{\mathcal C_i}$.
\end{enumerate}

A Mori dream space $Z$ with $\rho(Z)=2$ always has a \emph{$2$-ray configuration}, which is defined as follows. 
Let $M_1, M_2$ be $\mathbb{Q}$-divisors such that \[\Movb Z= \mathbb R_+[M_1]+ \mathbb R_+[M_2].\] Denote by $\varphi_1\colon Z\dashrightarrow Z_1$ (resp.~$\varphi_2\colon Z\dashrightarrow Z_2$) the ample model of $M_1+ \varepsilon M_2$ (resp.~$\varepsilon M_1+ M_2$) for an arbitrarily small positive rational number $\varepsilon$. Let $f_i \colon Z_i \to X_i$ be the ample model of $(\varphi_i)_*M_i$. Then, the birational map $\varphi_i$ is small, and $f_i$ is a fibration when $[M_i]$ lies on the boundary of $\Effb Z$, and $\varphi_i$ is a birational map that contracts a single exceptional divisor otherwise. These maps fit in a diagram which we call a $2$-ray configuration:
\begin{equation}\label{eq0} \xymatrix{
Z_1 \ar@{<--}[rr]^{\varphi_1} \ar[d]_{f_1} && Z \ar@{-->}[rr]^{\varphi_2}&& Z_2\ar[d]^{f_2}\\
X_1&&&&X_2
}\end{equation}
When $f\colon Z\to X$ is a divisorial contraction, we may assume that $\varphi_1$ is the identity map and that $X$ is the ample model of $M_1$ (i.e.~that $f=f_1$). 
\subsection*{The Sarkisov Program.} We recall a few notions on birational maps between end products of the classical MMP.  

\begin{dfn} Let $X$ be a terminal $\mathbb{Q}$-factorial variety, and $p\colon X\to S$ a morphism with positive dimensional fibres (so that $\dim S<\dim X$). Then $X/S$ is a \emph{Mori fibre space} if $p_{\ast} \mathcal O_X= \mathcal O_S$, ${-}K_X$ is $p$-ample and $\rho(X)= \rho(S)+1$.
\end{dfn}
The classical MMP shows that any uniruled terminal $\mathbb{Q}$-factorial variety $Z$ is birational to a Mori fibre space, so that $\mathcal P(Z)\neq \emptyset$. 
The {\it Sarkisov program} decomposes any birational map between Mori fibre spaces $[X/S], [X^\prime/S^\prime]\in \mathcal P(Z)$ into a finite number of {\it Sarkisov links} \cite{Co95, HM13}. 
Next, we recall the definition of Sarkisov links. 
\begin{dfn}A \emph{divisorial contraction} $f\colon Z \to X$ is a morphism between terminal $\mathbb{Q}$-factorial varieties such that $-K_{Z}$ is $f$-ample, $f_*\OO_Z= \OO_X$, and $\rho(Z) = \rho(X)+1$. 

We sometimes call $f$ an \emph{extraction} when we study properties of $f$ in terms of its target $X$. 
\end{dfn}

\begin{dfn}\label{sar}
  Let $X/S$ and $X^\prime/S^\prime$ be two Mori fibre spaces. A \emph{Sarkisov link} is a birational map $\varphi \colon X \dashrightarrow X^\prime$ of one of the following
  types:
  \begin{enumerate}[(I)]
  \item A \emph{link of type I} is a commutative diagram:
\[
\xymatrix{ &  Z\ar[dl]\ar@{-->}[r] & X^\prime\ar[d] \\
X\ar[d] &  & \ar[lld]S^\prime \\
S & &  
}
\]
where $Z\to X$ is a divisorial contraction and
$Z\dasharrow X^\prime$ a sequence of flips,
flops and inverse flips between terminal $\mathbb{Q}$-factorial varieties;
\item A \emph{link of type II} is a commutative diagram:
\[
\xymatrix{ &  Z\ar[dl]\ar@{-->}[r] & Z^\prime\ar[dr] & \\
X\ar[d] &  & & X^\prime \ar[d]\\
S\ar@{=}[rrr] & & & S^\prime
}
\]
where $Z\to X$ and $Z^\prime \to X^\prime$ are divisorial
contractions and $Z\dasharrow Z^\prime$ a sequence of flips,
flops and inverse flips between terminal $\mathbb{Q}$-factorial varieties;
\item A \emph{link of type III} is the inverse of a link of type I;
\item A \emph{link of type IV} is a commutative diagram:
\[
\xymatrix{X\ar[d]\ar@{-->}[rr] & & X^\prime\ar[d]\\
S\ar[dr] & & S^\prime\ar[dl]\\
&T & & 
}
\]
where $X\dasharrow X^\prime$ is a sequence of flips, flops and
inverse flips between terminal $\mathbb{Q}$-factorial varieties.
  \end{enumerate}
\end{dfn}

\begin{dfn}
Let $X/S$ be a Mori fibre space and $f\colon Z\to X$ an extraction; $f$ {\it initiates a link} if it fits into an Sarkisov link. 
\end{dfn}
The following lemma is a criterion for a divisorial extraction to initiate a link. It is of little practical use, but we want to highlight some of the subtleties that arise when proving that a $2$-ray configuration is indeed a Sarkisov link. 
\begin{lem}
Let $X$ be a terminal $\mathbb{Q}$-factorial Fano variety with $\rho(X)=1$ and let $f\colon Z\to X$ be an extraction. Then $f$ initiates a link if and only if the following hold:
\begin{enumerate}[(i)]
\item $Z$ is a Mori dream space,
\item if $\psi\colon Z\dashrightarrow Z^\prime$ is a small birational map and $Z^\prime$ is $\mathbb{Q}$-factorial, then $Z^\prime$ has terminal singularities, 
\item $[{-}K_{Z}]\in \inte (\Movb Z)$.
\end{enumerate}
\end{lem}
\begin{proof}
Assume that $f\colon Z \to X$ is a divisorial contraction that initiates a link. Then, as $X$ is a Fano $3$-fold with rational singularities, $h^1(X, \OO_X)=0$, and since $f_*\OO_Z= \OO_X$, by the Leray spectral sequence, $h^1(Z, \OO_Z)=0$ and we have the equality $\Pic(Z)_\mathbb{Q}= N^1_\mathbb{Q}(Z)$. 

Since $\rho(Z)=2$, if $f$ initiates a Sarkisov link, then following the notation of Definition~\ref{sar},  
there are $2$ distinct birational contractions $Z\to X$ and $Z\dashrightarrow X^\prime$ that are compositions of finitely many elementary maps (flips, flops, inverse flips and divisorial contractions) between terminal $\mathbb{Q}$-factorial varieties. Thus $Z$ has a $2$-ray configuration as above, and $Z$ is automatically a Mori dream space. The chambers of the decomposition of $\Effb Z$ are indexed by the divisorial contraction $Z\to X$ and by the elementary maps that decompose $Z\dashrightarrow X^\prime$.  

Furthermore, if $X^\prime$ is Fano, then 
\begin{equation}\label{eqa} {-}K_{Z}\in \mathbb R_+[f^*({-}K_X)]+ \mathbb R_+[(g\circ \psi)^*({-}K_{X^\prime})]= \Movb Z \end{equation}  and the class of ${-}K_Z$ is in the interior of $\Movb Z$ because $X$ and $X^\prime$ have terminal singularities and $Z\to X$ and $Z^\prime\to X^\prime$ are not isomorphisms.

If $X^\prime/S^\prime$ is a Mori fibre space with $\dim S^\prime \geq 1$, then $K_{X^\prime}= \psi_* K_Z$, and \begin{align}\label{eqb} {-}K_{Z}\in \mathbb R_+[f^*({-}K_X)]+& \mathbb R_+[\psi^*({-}K_{X^\prime}+A_{S^\prime})]\\\subsetneq &\mathbb R_+[f^*({-}K_X)]+ \mathbb R_+[\psi^*A_{S^\prime}]=  \Movb Z,\nonumber \end{align} where $A_{S^\prime}$ is the pullback of a suitable ample divisor on $S^\prime$ and the class of ${-}K_Z$ is in the interior of $\Movb Z$ because, as before, $X$ is terminal so that $\mathbb R_+[{-}K_Z]\neq \mathbb R_+[f^*({-}K_X)]$ and ${-}K_Z$ is big so that ${-}K_Z\not \in \mathbb R_+[\psi^*A_{S^\prime}]$. 

We have seen that $Z$ is a Mori dream space so that if $D$ is any movable $\mathbb{Q}$-divisor, the $D$-MMP terminates with a $\mathbb{Q}$-factorial semiample model for $D$ which we denote $Z_D$. The small birational map $Z\dashrightarrow Z_D$ factors $Z\dashrightarrow X^\prime$, therefore $Z\dashrightarrow Z_D$ is the composition of finitely many elementary contractions between terminal $\mathbb{Q}$-factorial varieties, and in particular, $Z_D$ is terminal. 
Let $Z\dashrightarrow Z^\prime$ be an arbitrary small birational map, and assume that $Z^\prime$ is $\mathbb{Q}$-factorial. Let $D$ be the proper transform of an ample $\mathbb{Q}$-Cartier divisor on $Z^\prime$; then $D$ is mobile because $Z\dashrightarrow Z^\prime$ is small. By construction, $Z^\prime\simeq \Proj R(Z,D)$ is the ample model of $D$, and if $Z_D$ is the end product of a $D$-MMP on $Z$, then $Z_D\to Z^\prime$ is a morphism and a small map. Since $Z^\prime$ and $Z_D$ are both $\mathbb{Q}$-factorial, it follows that they are isomorphic, so that $Z^\prime$ has terminal singularities.   

Conversely, if $Z$ is a Picard rank $2$ Mori dream space then 
\[ \Movb Z= \mathbb R_+[M_1]+ \mathbb R_+[M_2] \subset \Effb(Z)= \mathbb R_+[D_1]+ \mathbb R_+[D_2]\] for effective $\mathbb Z$-divisors $M_1,M_2, D_1$ and $D_2$. Since $f\colon Z\to X$ is a divisorial extraction, $f$ is the ample model of $M_1\neq D_1$; note that $D_2$ needs not be distinct from $M_2$. 
When $M_2\neq D_2$, let $Z^\prime$ be the ample model of $M_2+ \varepsilon M_1$ for an arbitrarily small positive rational number $\varepsilon$ and let $X^\prime$ be the ample of model of $M_2$. Then, $Z \dashrightarrow Z^\prime$ is a small birational map, and $Z^\prime$ has terminal singularities by assumption (ii). The birational map $Z \dashrightarrow Z^\prime$ is small, hence we may identify divisors on $Z$ and on $Z^\prime$ and, under this identification, $\Movb Z$ is equal to $\Movb Z^\prime$, so that $[{-}K_{Z^\prime}]$ is in the interior of $\Movb Z^\prime$ by assumption (iii). It follows that the morphism $Z^\prime \to X^\prime$ is $K$-negative and that $X^\prime$ has terminal singularities. When $M_2=D_2$, let $X^\prime$ be the ample model of $M_2+ \varepsilon M_1$ for an arbitrarily small positive rational number $\varepsilon$ and let $S^\prime$ be the ample model of $M_2$. Then, $X^\prime$ is terminal by  assumption (ii) and the fibration $X^\prime\to S^\prime$ is $K$-negative by assumption (iii). Since $M_2$ is not a big divisor, $\dim S^\prime< \dim X^\prime$ and $X^\prime/S^\prime$ is a Mori fibre space. 
\end{proof}
\begin{rem}
Condition (ii) may only fail when the $2$-ray configuration on $Z$ involves an antiflip because flips and flops of terminal varieties are automatically terminal. 
For example, consider \[Z= \mathbb{P}(\OO_{\mathbb{P}^1}\oplus\OO_{\mathbb{P}^1}(-2)\oplus \OO_{\mathbb{P}^1}(-2)),\] then $Z$ is a Mori fibre space (a $\mathbb{P}^2$-bundle over $\mathbb{P}^1$) and a Mori dream space on which (i), (iii) hold but (ii) fails. It follows that the $2$-ray configuration on $Z$ does not produce a Sarkisov link.
Example~\ref{antiflip} is a Sarkisov link involving an antiflip, and we check that condition (ii) holds directly. 
\end{rem}
\begin{rem}
Note that condition (ii) always holds when $M_2$ is of the form $K_Z+ \Theta$ for $\Theta$ a nef divisor. Indeed, in that case, every $D\in \Movb Z\cap \B Z$ is of the form $K_Z+ \Theta^\prime$ for $\Theta^\prime$ nef, and every $D$-negative birational contraction is $K$-nonpositive (see \cite[2.10]{Kal12}). In particular, the ample model $\varphi_D\colon Z\dashrightarrow Z_D$ is a $D$-negative birational contraction, hence is $K_Z$-nonpositive, so that for any resolution $(p,q)\colon U \dashrightarrow Z\times Z_D$, we have:
\[ p^*K_Z= q^*K_{Z_D} + E,\] 
 where $E$ is an effective $q$-exceptional divisor. This implies that for any divisor $F$ over $Z_D$, the discrepancy $a_F(Z_D)\geq a_F(Z)$,  and that $Z_D$ has terminal singularities if $Z$ does.
\end{rem}

\section{Terminal singularities on quartic $3$-folds}\label{terminal}
In this section, we recall some results on the local analytic description of terminal hypersurface singularities in dimension $3$ and we bound the local analytic type of singularities on terminal factorial quartic hypersurfaces in $\mathbb{P}^4$.  

\subsection{Local analytic description and divisorial extractions}We first recall a few results on isolated hypersurface singularities.
\subsection*{Singularity theory}
Let $\mathbb C[[x_1, \cdots , x_n]]$ be the ring of complex formal power series in $n$ variables, and $\mathbb C\{ x_1, \cdots , x_n\}\subset \mathbb C[[x_1, \cdots, x_n]]$ the subring of formal power series with nonzero radius of convergence. For $F\in \mathbb C\{x_1, \cdots x_n\}$, $(F=0)$ is a germ of a complex analytic set, and the {\it singularity} $(F=0)$ is the scheme $\Spec_{\mathbb C} \mathbb C[[x_1, \cdots x_n]]/(F)$. 

Let $F\in \mathbb C[[x_1, \cdots , x_n]]$ be a power series, and $d$ a positive integer. We denote by $F_d$ the degree $d$ homogeneous part of $F$ and by $F_{\geq d}$ the series $\sum_{k\geq d} F_{k}$. The {\it multiplicity} of $F$ is $\mult F= \min\{ d\in \mathbb N \vert F_d \neq 0\}$.  
Two power series $F,G$ are {\it equivalent} if there exist an automorphism $\varphi= (\varphi_1, \cdots , \varphi_n) \in \Aut (\mathbb C[[x_1, \cdots x_n]])$ and a unit $u\in (\mathbb C[[x_1, \cdots x_n]])^*$ such that 
\[ u(x_1, \cdots , x_n) G(x_1, \cdots , x_n)= F(\varphi_1, \cdots , \varphi_n).\]

In other words, $F$ and $G$ are equivalent if the singularities $(F=0)$ and $(G=0)$ are isomorphic. We denote the equivalence of power series by $F\sim G$. 

In what follows, as we are only interested in isolated critical points, by \cite[I, Section 6.3]{AGV} we may (and will) assume that the power series $u, F$ and $G$ all have nonzero radius of convergence and that $\varphi \in \Aut(\mathbb C\{x_1, \cdots x_n\})$. 
 \begin{dfn}\
 \begin{itemize}
 \item[1.]
The singularity $(h(x, y,z)=0)$ is {\it Du Val} if $h$ is equivalent to one of the standard forms:
\begin{eqnarray*}
A_n& \quad x^2 + y^2 + z^{n+1} =0 \mbox{ for } n\geq 0\\
D_m&  \quad x^2 + y^2z + z^{m-1} =0 \mbox{ for } m\geq 4\\
E_6&  \quad x^2+ y^3+ z^4 =0\\
E_7& \quad x^2+ y^3+ yz^3 =0\\
E_8& \quad x^2+ y^3+ z^5 =0
\end{eqnarray*}
\item[2.]
A singularity $(F(x,y,z,t)=0)$ is {\it compound Du Val} if $F$ is equivalent to: 
\begin{equation}\label{cDV}h(x, y,z)+ tf(x,y,z,t)=0\end{equation}
where $(h=0)$ is a Du Val singularity. 
The singularity $(F=0)$ is $cA_n$ (resp.~$cD_m, cE_k$) if in \eqref{cDV}, $h$ is of type $A_n$ (resp.~$D_m, E_k$) with $n$ (resp.~$m,k$) minimal. 
\end{itemize}
\end{dfn}
\begin{thm}\cite[Theorem 2.8]{Kol98}
Let $F(x, y,z,t)\in \mathbb C[[x,y,z,t]]$ define a $cA$ singularity, then one of the following holds:\begin{enumerate}[]
\item
{\bf $cA_0$:} $F\sim x$,
\item {\bf $cA_1$:} $F\sim x^2+y^2+z^2+ t^m$ for $m\geq 2$,
\item {\bf $cA_n, n\geq 2$:} $F\sim x^2+y^2+f_{\geq n+1}(z,t)$, where $f_{\geq n+1}(z,t)$ has no multiple factors.
\end{enumerate}\end{thm}

\begin{rem}
Up to change of coordinates on $\mathbb{P}^1_{z,t}$, we may assume that $z^{n+1}$ appears with coefficient $1$ in $f_{\geq n+1}(z,t)$. Since $(F=0)$ is an isolated singularity,  $f_{\geq n+1}(z,t)$ has no repeated factor and contains at least one monomial of the form $t^N$ or $zt^{N-1}$ for $N\geq n+1$. When  $N>n+1$, as in \cite[Section 12]{AGV}, \[ F\sim x^2+y^2+z^{n+1}+ t^N \mbox{ or } F\sim x^2+y^2+z^{n+1}+ zt^{N-1}\] for $N\geq n+1$.
Similarly, when $N=n+1$, we have \[F\sim x^2+y^2+ f_{n+1}(z,t)\] where $f_{n+1}$ is a homogeneous form with no repeated factor of degree $n+1$. 
 \end{rem}
\begin{thm} \cite[Theorem 2.9]{Kol98} \label{tcD}Let $F(x, y,z,t)\in \mathbb C[[x,y,z,t]]$ define a $cD$ singularity, then one of the following holds:
\begin{enumerate}[]
\item {\bf $cD_4$:} $F\sim x^2+ f_{\geq 3}(y,z,t)$, where $f_3$ is not divisible by the square of a linear form,
\item{\bf $cD_{>4}$:} $F\sim x^2+ y^2z+ayt^r+ h_{\geq s}(z,t)$, where $a\in \mathbb C$, $r\geq 3, s\geq 4$ and $h_s\neq 0$. This has type $cD_m$ for $m= \min\{2r, s+1\}$ if $a\neq 0$ and $m=s+1$ otherwise. 
\end{enumerate}\end{thm}

\begin{dfn}
The {\it Milnor number} of the singularity $(F=0)$ is
\[ \mu (F=0)= \dim_{\mathbb C} \mathbb C[[x_1, \cdots , x_n]]/J_F,\]
where $J_F= (\frac{\partial F}{\partial x_1}, \cdots ,\frac{\partial F}{\partial x_n})$ is the Jacobian ideal of $F$, i.e.~the ideal generated by the partial derivatives of $F$. 

If $F\sim G$, the Milnor numbers of $(F=0)$ and of $(G=0)$ are equal. The Milnor number $\mu (F=0)$ is finite precisely when $(F=0)$ is an isolated singularity.  
\end{dfn}

\begin{lem}\label{lem1}
\begin{enumerate}[1.]
\item
If $(F=0)$ is a $cA_n$ singularity with $n\geq 1$ then $\mu(F=0)\geq n^2$.  
If $F\sim  xy+ z^{n+1}+ t^N$  or $F\sim  xy+ z^{n+1}+ zt^{N-1}$ for $N>n+1$ then $\mu(F=0)\geq n(N-1)$. \item If $(F=0)$ is a $cD_m$ singularity with $m>4$ then $\mu(F=0) \geq m(m-2)$.
\end{enumerate}
\end{lem}

\begin{proof}
If $(F=0)$ is a $cA_n$ singularity with $n\geq 1$, then $F\sim xy+ f(z, t)$, where $f(z,t)$ has multiplicity greater than or equal to $n+1$. Since 
\[ \mu(F=0)= \mu (xy+ f(z,t)=0), \mbox{ we have } \mu(F=0)= \dim_\mathbb C \mathbb C[[z, t]]/J_f.\]

In all cases, since $\mu(F=0)$ is finite and $f$ has no repeated factor, $\mathbb C[[z,t]]/J_f$ has dimension $\deg \frac{\partial f}{\partial z}\cdot \deg \frac{\partial f}{\partial z} \geq n(N-1)$. 

Now assume that $(F=0)$ is a $cD_m$ singularity with $m>4$. Then, as in \cite[I, Section 12]{AGV}, if $F_0$ is the quasihomogeneous part of $F$ then $\mu(F=0)= \mu(F_0=0)$, and $\mu(F_0=0)$ is given by the formula \cite[I, Corollary 3, p.~200]{AGV}. In the notation of Theorem~\ref{tcD}, using the methods of \cite[I, Section 12]{AGV}, we obtain that: 
\begin{enumerate}[(a)]
\item either $F_0 \sim x^2+y^2z+yt^r+ z^s$ or $F_0 \sim x^2+y^2z+yt^r+ z^{s-1}t$ and $\mu(F_0=0)\geq m(m-2)$,
\item or $F_0\sim x^2+ y^2z+ h(z, t)$, where $h(z,t)$ is of the form $z^s+t^N$, $z^{s-1}t+ t^N$, $z^N+ t^s$ or $z^N+ zt^{s-1}$ for some $N\geq s$. In all cases, $\mu(F_0=0) \geq m(m-2).$
\end{enumerate}
\end{proof}

\subsection{Bounding the local analytic type of singularities on a terminal factorial Fano Mori fibre space}
We bound the local analytic type of singularities on a terminal factorial Fano $3$-fold with Picard rank $1$.  

\begin{thm} \cite{Nam97a}\label{smoothing}
Let $X$ be a Fano $3$-fold with terminal Gorenstein singularities. Then $X$ has a smoothing, i.e.~there is a one parameter flat deformation:
\[ \xymatrix{X \ar@{^{(}->}[r] \ar[d] &\mathcal X\ar[d]\\ \{0\} \ar@{^{(}->}[r] &\Delta}\]
with $\mathcal X_t$ smooth when $t\neq 0$. For all $t\in \Delta$, $\mathcal X_t$ is Fano, $\rho(X)= \rho(\mathcal X_t)$ and ${-}K_X^3={-}K_{\mathcal X_t}^3$. \end{thm}
The existence of a smoothing $X\hookrightarrow \mathcal X$ allows us to bound the Milnor numbers of singularities on $X$. 

\begin{thm}\cite[Theorem 3.2]{NS95}\label{LES}
Let $X$ be a normal projective $3$-fold with isolated rational hypersurface singularities such that $H^2(X, \OO_X)=(0)$. Denote $b_i(X)$ the $i$-th Betti number for the singular cohomology of $X$. If there is a smoothing $X \hookrightarrow \mathcal X$, then 
\[ b_4(X)-b_2(X)= b_3(X)- b_3(\mathcal X_t)+ \sum_{P\in \Sing X} \mu(X, P)\]
for $t\in \Delta \smallsetminus \{0\}$, where $\mu(X, P)$ is the Milnor number of $(X, P)$.  
\end{thm}
When $X$ is a terminal and factorial Fano $3$-fold, the second and fourth Betti numbers of $X$ are equal, that is $b_2(X)=b_4(X)$, 
so that: 
\begin{equation}\label{eq1}
\sum_{P\in \Sing X} \mu(X, P)= b_3(\mathcal X_t)-b_3(X)\leq b_3(\mathcal X_t).
\end{equation}
The third Betti numbers of non-singular Fano $3$-folds with Picard rank $1$ are known (see \cite[Table 12.2]{IP99}), and we obtain a bound on the sum of Milnor numbers of singular points on $X$ that only depends on ${-}K_X^3$. When ${-}K_X^3=4$, we have the following. 
\begin{pro}\label{pro1} Let $X\subset \mathbb{P}^4$ be a terminal factorial quartic hypersurface.  
If $(P\in X)$ is a singular point of type $cA_n$, then $n$ is at most $7$. 

If $(P\in X)$ is a singular point of type $cD_m$, then $m$ is at most $8$. 
 \end{pro}
  \begin{proof} Let $X \hookrightarrow \mathcal X$ be a smoothing, then for all $t\neq 0$, $\mathcal X_t$ is a nonsingular quartic hypersurface and $b_3(\mathcal X_t) =60$ (see \cite[Table 12.2]{IP99}). By Theorem~\ref{LES}, we have $\mu(X, P)$ is bounded above by $60$, and the result follows immediately from the lower bounds obtained in Lemma~\ref{lem1}.\end{proof}
   \begin{rem}\label{sharp}
The bound on the local analytic type of $cA$ points is sharp, Example~\ref{cA7} is an example of a terminal factorial quartic hypersurface with a $cA_7$ singular point. 

We do not believe that the bound on the local analytic type of $cD$ points is optimal, as we have not been able to write down examples attaining it.
We give examples of terminal factorial quartic hypersurfaces with isolated singular points of type $cD_4, cD_5$ and $cE_6, cE_7$ and $cE_8$ in Section~\ref{cD,cE}.

 \end{rem}
 \begin{rem} By the classification of non-singular Fano $3$-folds, the bounds on the local analytic type of singularities lying on a terminal factorial Fano $3$-fold with Picard rank $1$ and anticanonical degree $-K^3>4$ are even more restrictive than in the case of a quartic hypersurface. \end{rem}

\begin{rem} We can use the same methods to bound the local analytic type of singularities on an arbitrary terminal Gorenstein Fano $3$-fold $X$ with $\rho=1$. Indeed, by Theorem~\ref{LES}: 
\[ \sum_{P\in \Sing X} \mu(X, P)\leq b_3(\mathcal X_t)+ \sigma(X),\]
where $\sigma(X)= b_4(X)-1$ is the {\it defect} of $X$.  The defect of terminal Gorenstein Fano $3$-folds with $\rho=1$ is bounded in \cite{Kal07b}. For example, if $X$ is a (not necessarily factorial) terminal quartic hypersurface, then 
\[ \sum_{P\in \Sing X} \mu(X, P)\leq 75,\]
so that if $(P\in X)$ is a $cA_n$ point, $n\leq 8$. 

Our main interest in this article is in non-rigid quartic $3$-folds; in particular, we have not tried to write down examples of non-factorial quartic hypersurfaces with a singular point of type $cA_8$. We believe that such an example would be found with extra work. 
\end{rem}

\subsection{Divisorial extractions with centre at a $cA_n$ point}
Kawakita classifies the germs of divisorial extractions $f\colon Z\to X$ with centre at a $cA_n$ point. We recall this classification here, as we will use it in Section~\ref{examples}. 
\begin{thm}\cite{Kawk02, Kawk03}
\label{cAn}
Let $(P\in X)$ be a $cA_n$ point and $f\colon Z\to X$ a divisorial extraction with $f(\Exc f)= \{P\}$. Then, one of the following holds:
\begin{enumerate}[1.]
\item {\bf General type:} under a suitable local analytic identification
\[ (P\in X) \simeq 0\in\{ xy+g(z,t)=0\}\subset \mathbb C^4,\]
and $f$ is the blowup of $\mathbb C^4_{x,y,z,t}$ with weights $(r_1,a(n+1)-r_1,a,1)$, where $a$ is the discrepancy of $f$, $a$ and $r_1$ are coprime integers, and \[g(z,t)= z^{n+1}+ g_{\geq a(n+1)}(z, t)\] has weighted degree $a(n+1)$.
In particular,
\[ E= \Exc f \simeq \{xy+ g_{a(n+1)} (z,t)=0\} \subseteq \mathbb{P}(r_1, a(n+1)-r_1, a, 1),\]
and if we denote by $r_2= a(n+1)-r_1$, then we have $aE^3= 1/r_1+1/r_2$. 
\item {\bf Exceptional type, $n=1$:} under a suitable local analytic identification
\[ (P\in X) \simeq 0\in\{ xy+z^2+t^3=0\}\subset \mathbb C^4,\]
 and $f$ is the blowup with weights $(1,5,3,2)$, $f$ has discrepancy $a=4$. 
 \item {\bf Exceptional type $n=2$:}  under a suitable local analytic identification
\[ (P\in X)\simeq 0\in\{ xy+z^3+ g_{\geq 4}(z,t)=0\}\subset \mathbb C^4,\]
the discrepancy of $f$ is $3$ and $Z$ has exactly one non-Gorenstein point 
\[(Q\in Z)\simeq \{0\}\in \{x^2+y^2+z^2+t^2=0\}\subset \mathbb C^4/\frac{1}{4}(1,3,3,2).\] 
\end{enumerate} 
\end{thm}
\begin{rem} Let $X\subset \mathbb{P}^4$ be a terminal factorial quartic hypersurface. Then, the discrepancies of possible divisorial extractions $f\colon Z\to X$ with centre at a $cA_n$ point $(P\in X)$ can be bounded in the same way as in Proposition~\ref{pro1}. Indeed, by Lemma~\ref{lem1}, if there is an extraction of general type in Theorem~\ref{cAn}, the Milnor number $\mu(X, P)$ satisfies \[\mu(X,P)\geq n(a(n+1)-1),\] but by Theorem~\ref{LES}, $\mu(X,P)\leq 60$. It follows that the discrepancy of $f$ satisfies:
\begin{table}[h] 
\centering 
\begin{tabular}{|c|c|c|c|c|c|}
\hline
n& $6,7$& $5$ & $4$& $3$& $2$\\
\hline
a & $1$& $\leq 2$& $\leq 3$& $\leq 5$ & $\leq 11$\\
\hline
\end{tabular}
\end{table} 
\end{rem}

\section{Examples of non-rigid terminal quartics}\label{examples}
In this section, we present examples of non-rigid terminal factorial quartic hypersurfaces $X\subset \mathbb{P}^4$. Each of these examples has a Sarkisov link initiated by an extraction $f\colon Z\to X$ that contracts a divisor to a singular point. In most of our examples, the singular point is of type $cA_n$, so that the germ of $f$ is a weighted blowup as in Kawakita's classification (Theorem~\ref{cAn}). Our examples are obtained by \emph{globalising} these germs: we can write down an explicit description of $f$ in projective coordinates. 

Let $X\subset \mathbb{P}^4$ be a terminal factorial quartic hypersurface, and assume that $P=(1{:}0{:}0{:}0{:}0)\in X$ is a $cA_n$ point. Up to projective change of coordinates, the equation of $X$ can be written:
\[X=\{\varphi_4(x_0, \cdots , x_4)=  x_0^2x_1x_2+ x_0 \psi_3(x_1, \cdots, x_4)+ \theta_4(x_1, \cdots , x_4)=0\}, \]
where $\varphi_4$ is a homogeneous polynomial of degree $4$ in the variables $x_0, \cdots , x_4$ and $\psi_3$ and $\theta_4$ are homogeneous polynomials in the variables $x_1, \cdots, x_4$. 

The first step in our construction is to look for examples of hypersurfaces $X\subset \mathbb{P}^4$ that can be embedded in a larger weighted projective space $\mathbb{P}$ in such a way that the restriction of a suitable weighted blowup $F\colon \mathcal F\to \mathbb{P}$ is a divisorial contraction $f\colon Z\to X$. We take $X$ to be a complete intersection of the form:
\[X=\begin{cases} \alpha = x_0x_1 + q(x_3, x_4)\\ \beta = x_0x_2 + q'(x_3, x_4)\\ \alpha \beta + \varphi_4'(x_0, \cdots x_4)= 0\end{cases}\subset \mathbb{P}(1^5, 2^2)= \mathbb{P},\]
for homogeneous forms $\alpha, \beta$ of degree $2$. The equation of the hypersurface $X\subset \mathbb{P}^4$ is recovered by substituting $\alpha, \beta$ in the third equation.

Explicitly, we want the germ of $f\colon Z\to X$ to be a weighted blowup of general type in the classification of Theorem~\ref{cAn}. This means that up to local analytic identification, denoting by $a$ the discrepancy of $f$, we have
\begin{equation}\label{eq4} (P\in X) \simeq 0\in\{ xy+g(z,t)=0\}\subset \mathbb C^4,\mbox{ where } g(z,t)= z^{n+1}+ g_{\geq a(n+1)}(z, t)\end{equation} and $g$ has weighted degree $a(n+1)$.

We choose $\varphi_4$ so that for suitable $\alpha, \beta$, setting $x_0=1$ in the third equation gives: 
\[ \alpha\beta+ g(x_3, x_4) + (\mbox{higher weighted order terms})=0.\]
In other words, the restrictions of $\alpha, \beta$ to $\{ x_0=1\}$ define the local analytic isomorphisms that bring the equation of $X\cap \{x_0=1\}$ into the desired form \eqref{eq4}.
The divisorial contraction $f\colon Z\to X$ we construct are restrictions of weighted blowups $F\colon \mathcal F \to \mathbb{P}(1^5,2^2)$, where the weights assigned to the variables $\alpha, \beta, x_3, x_4$ are as in Theorem~\ref{cAn}.
 
The second step in our construction is to show that some of these divisorial contractions initiate Sarkisov links. Since $\mathcal F$ is a Mori dream space (it is toric), it has a $2$-ray configuration as in \eqref{eq0}. 
We check directly that the $2$-ray configuration 
\begin{align}\label{confi}  \xymatrix{Z\subset \mathcal F\ar@{-->}[rr]^{\Phi} \ar[d]_F && \mathcal F+\ar[d]^{F^+}\\
X\subset \mathbb{P} && \mathbb{P}^+} \mbox{ restricts to a Sarkisov link } \xymatrix{Z \ar@{-->}[r] \ar[d]& Z^+\ar[d]\\ X & X^+},\end{align}
or that we can find another embedding $Z\subset \mathcal F^\prime$ via unprojection such that the $2$-ray configuration on $\mathcal F^\prime$ restricts to a Sarkisov link. 
Note that we do not make any assumption on the singularities of $\mathcal F$, in particular, $\mathcal F$ needs not be terminal and $\mathbb{Q}$-factorial. We check the following:
\begin{enumerate}[1.]
\item The map $\Phi_{\vert Z}$ is an isomorphim in codimension $1$ and $Z^+$ is terminal. In our examples,  $\Phi_{\vert Z}$ is the composition of finitely many elementary maps
\[Z= Z_0 \stackrel{\varphi_0}\dashrightarrow Z_1\stackrel{\varphi_1} \dashrightarrow \cdots \stackrel{\varphi_n} \dashrightarrow Z_{n+1}=Z^+\] that are isomorphisms or antiflips, flops and flips (in that order). If $Z_i \stackrel{\varphi_i}\dashrightarrow Z_{i+1}$ is $K$-nonpositive (e.g.~a flip or flop) and $Z_i$ is terminal and $\mathbb{Q}$-factorial, then so is $Z_{i+1}$. We need to check directly that antiflips preserve the terminal condition; we do this by identifying the antiflips as inverses of flips appearing in \cite{Br99}. 
\item The restriction of $\mathcal F^+\to \mathbb{P}^+$ is the contraction of a $K$-negative extremal ray $Z^+ \to X^+$. 
\end{enumerate} 
 
We always denote by $P$ the point $(1{:}0{:}0{:}0{:}0)\in \mathbb{P}^4$.  
\begin{rem}\label{fac}
We do not give details of how to check that each of our examples is factorial. This relies on ad hoc methods and the general scheme is as follows. As $X\subset \mathbb P$ is a Picard rank $1$ hypersurface, $X$ is ($\mathbb Q$-)factorial precisely when the rank of the divisor class group satisfies \[\mathrm{rk} \mathrm{Cl}(X)= \rho(X)= \mathrm{rk} \mathrm{Cl}(\mathbb P).\] Since $\Sing X$ has codimension $3$, we have an isomorphism 
\[ \mathrm{Cl}(X)\simeq \mathrm{Cl}(X\smallsetminus \Sing(X)).\]
Let $\pi \colon \tilde{\mathbb P}\to \mathbb P$ be a map from a smooth variety $\tilde{\mathbb P}$ that restricts to a resolution $\tilde X \to X$ with exceptional locus $E_{X}$. We have natural isomorphisms
\[ \mathrm{Cl}(X)\simeq  \mathrm{Cl}(X\smallsetminus \Sing(X))\simeq \Pic (X \smallsetminus \Sing X)\simeq \Pic(\tilde X\smallsetminus E_{X}) \]
and similarly \[\mathrm{Cl}(\mathbb P)\simeq \Pic (\tilde{\mathbb P}\smallsetminus \Exc \pi).\] Now, $X$ is factorial precisely when $\Pic (\tilde{\mathbb P}\smallsetminus \Exc \pi)\simeq \Pic (\tilde X \smallsetminus E_{\tilde X})$. As the classical Grothendieck-Lefschetz theorem guarantees that $\Pic (\tilde X)\simeq \Pic (\tilde{\mathbb P})$, the result follows by comparison of the kernels of the surjective maps \[ r_1\colon \Pic (\tilde{\mathbb P}) \to \Pic (\tilde{\mathbb P}\smallsetminus \Exc \pi) \mbox{ and } r_2\colon \Pic (\tilde X) \to \Pic (\tilde X\smallsetminus E_{X}).\] These kernels are isomorphic to the free abelian groups on irreducible divisorial components of $\Exc \pi$ and of $E_X$ respectively, and can be worked out in each case. 

Note that we have $2$-ray configuration \eqref{confi}, and therefore all varieties in \eqref{confi} are $\mathbb Q$-factorial if and only if one of them is. These varieties are complete intersections in $\mathbb P, \mathcal F, \mathcal F^+$ and $\mathbb P^+$ and their Picard ranks are $1,2,2$ and $1$ by construction, and the method above can be applied to determine the divisor class group of any of them. In some cases, it can be easier to determine $\mathbb Q$-factoriality of another variety in \eqref{confi} than $X$ (one needs to keep track of Gorenstein indices). \end{rem}
\subsection{A family of examples.}
\label{family}
In this section, we study some links initiated by divisorial contractions with discrepancy $1$. 
We consider the family of quartic hypersurfaces: 
\[ X^{i,j}_4=\{(x_0x_1-x_3^2)^2-(x_0x_2-x_4^2)^2+x_0^{4-i-j}
(x_1^ix_3^j+x_2^ix_4^j)+x_1^4+x_2^4=0 \} \subset \mathbb{P}^4, \]
where $(i,j)$ satisfy $3\leq i+j\leq 4$ and $(i,j) \neq (4,0)$.
Then $(P\in X^{i,j})$ is a $cA_n$ point with: 
\begin{table}[ht] \label{table1}
\[\begin{array}{c||c|c|c|c||c|c|c|c}
(i,j)&(0,4)&(1,3)&(2,2)&(3,1)&(0,3)&(1,2)&(2,1)&(3,0)\\
\hline
n & 3 & 4 & 5 & 6 & 2 & 3 & 4 & 5
\end{array}\]
\caption{}
\end{table}

Indeed, the restriction of $X^{i,j}$ to the affine chart $U_0=\{x_0=1\}$ is 
\[ \{F(x,y,z,t)=(x-z^2)^2-(y-t^2)^2+ x^iz^j+y^it^j +x^4+y^4=0\}\subset \mathbb C^4,\]
and the singularity $(F=0)$ is equivalent to $(G=0)$, where $G$ defines a $cA_n$ singularity. To see this, let $\underline{\varphi}=(\varphi_1, \varphi_2, \varphi_3, \varphi_4) \in \Aut \mathbb C [[x, y,z,t]]$ be defined by:
\[ 
\begin{cases}\varphi_1(x,y,z,t) = x+y-(z^2+t^2)\\ \varphi_2(x,y,z,t) = x-y-(z^2-t^2)\\\varphi_3(x,y,z,t)=z \\\varphi_4(x,y,z,t)=t \end{cases}\]
and check that $F(x,y,z,t)= G(\varphi_1, \varphi_2, \varphi_3, \varphi_4)$, where 
\begin{equation}\label{eq2} G(x,y,z,t)= xy+ z^{n+1}+ t^{n+1}+ G'(x,y,z,t),\end{equation}
where $n$ is as in Table 1, and $G'(x,y,z,t)$ is a polynomial of degree strictly greater than $1$ with respect to the weights \[\underline w=(1/2, 1/2, 1/n+1,1/n+1).\] Since $G$ is a semi quasi-homogeneous polynomial of degree $1$ with respect to $\underline w$, and since no element of its Jacobian algebra has degree strictly greater than $1$, by \cite[I, Section 12]{AGV}, $G\sim xy+ z^{n+1}+ t^{n+1}$ and $(P\in X)$ is a $cA_n$ point.  
The quartic hypersurfaces $X^{i,j}$ are terminal ($\Sing X^{i,j}=\{P\}$) and factorial,  and hence are Mori fibre spaces. 

\begin{rem}Taking $(i,j)=(4,0)$ gives a terminal quartic hypersurface with a $cA_7$ point
\[X=X^{4,0}=\{(x_0x_1-x_3^2)^2-(x_0x_2-x_4^2)^2+x_1^4+x_2^4=0\}\subset \mathbb{P}^4,\] 
but $X=\{f=0\}$ is not factorial as $f= q_1q_1'+ q_2q_2'$, where $q_i, q_i'$ are quadric polynomials. 
 \end{rem}

\begin{nt}\label{nt}
We embed $X^{i,j}$ as a complete intersection in a scroll whose coordinates are those of $\mathbb{P}^4$ on the one hand, and projectivisations of the (non-linear) components of $\underline \varphi\in \Aut \mathbb C[[x,y,z,t]]$, the automorphism we used to transform the equation of $X\cap U_0$ into $G$. Here, $\varphi_3(x,y,z,t)= z$ and $\varphi_4(x,y,z,t)= t$, so we only need to introduce the coordinates:
\[ \begin{cases} \alpha = x_0^2\varphi_1(\frac{x_1}{x_0},\frac{x_2}{x_0}, \frac{x_3}{x_0}, \frac{x_4}{x_0})= x_0(x_1+x_2)-(x_3^2+x_4^2)\\ \beta= x_0^2\varphi_2(\frac{x_1}{x_0},\frac{x_2}{x_0}, \frac{x_3}{x_0}, \frac{x_4}{x_0}))= x_0(x_1-x_2)-(x_3^2-x_4^2)\end{cases}.\]
By construction, $X$ is the complete intersection:
\[ X=\begin{cases}\alpha+ \beta= x_0x_1-x_3^2\\ \alpha-\beta = x_0x_2-x_4^2\\\alpha \beta+ x_0^{4-i-j}(x_1^ix_3^j+ x_2^ix_4^j) + x_1^4+ x_2^4=0 \end{cases}\subset \mathbb{P}(1^5, 2^2)\]
\end{nt}
We set $x_0=1$ to study $X$ in the neighbourhood of $P$, and we use the first two equations to eliminate $x_1, x_2$, so that $X\cap \{x_0=1\}$ is the hypersurface:
\[ \{\alpha \beta + x_3^{n+1}+ x_4^{n+1} + (\mbox{higher order terms})=0\}\subset \mathbb C^4_{x_3,x_4, \alpha, \beta}\]

We now construct Sarkisov links initiated by divisorial contractions $f\colon Z\to X$ that are restrictions of weighted blowups $F\colon \mathcal F\to \mathbb{P}(1^5, 2^2)$ with centre at $P$. We assume that the local analytic coordinates for which the germ of $f$ has the description given in Theorem~\ref{cAn} are $(\alpha, \beta, x_3, x_4)$. 
We use the notation of \cite[\S~3.1]{A13}; $\mathcal F$ is the Picard rank $2$ toric variety $\TV(I, A)$, where \[I=(u,x_0)\cap (x_1, \dots , x_4, \alpha, \beta)\] is the irrelevant ideal in the the coordinates ring $\mathbb C[u, x_0, \dots, x_4, \alpha, \beta]$ and for suitable nonzero integers $r_1, r_2, w_1, w_2$, the $2\times 8$ matrix $A$ below (the numerical rows) defines the action of $\mathbb C^\ast \times \mathbb C^\ast$ with weights 
\begin{equation}\label{eq3}
A=\left( \begin {array}{cccccccc}
 u & x_0& x_1&x_2 & x_3&x_4& \alpha & \beta\\
1 & 0 & -w_1 & -w_2 &-a &-1 & -r_1& -r_2\\
 0 & 1 & 1& 1& 1& 1 &2 &2\end{array}
\right).
\end{equation}
For example $(\mu,\lambda)\in\mathbb{C}^*\times\mathbb{C}^*$ acts on the variable $\alpha$ by
\[\big((\mu,\lambda),\alpha\big)\mapsto \mu^{-r_1}\lambda^2\alpha.\]
With the grading defined in \eqref{eq3}, taking the ample model of a divisor whose class is in $\left(\begin{array}{c}0\\ k\end{array}\right)$
\[\mathcal F \to \Proj\bigoplus_{k\in\mathbb{N}}H^0(\mathcal{F},\mathcal{O}_\mathcal{F}\left(\begin{array}{c}0\\ k\end{array}\right))\]
is the morphism given by: 
\[(u,x_0,x_1,x_2,x_3,x_4,\alpha,\beta)\mapsto (x_0, u^{\omega_1}x_1,u^{\omega_2}x_2,u^{a}x_3,ux_4,u^{r_1}\alpha,u^{r_2}\beta).\]
This is precisely the weighted blowup $F:\mathcal{F}\rightarrow\mathbb{P}(1^5,2^2)$ we are after.

In what follows, we always denote by 
$L$ the pullback of $\OO_{\mathbb{P}}(1)$-- so that $L \in \left( \begin{array}{c}0\\1\end{array}\right)$-- and by $E=  \left( \begin{array}{c}1\\0\end{array}\right)$ the $F$-exceptional divisor.
 L

The form of the expression \eqref{eq2} imposes that the discrepancy of $f$ is $1$. We have $-K_Z= f^*(-K_X)- E$, where $E= \Exc f$, so that $a=1$ and $r_1+ r_2= n+1$, where $(n, (i,j))$ are as in Table 1. Set $r_1=r$, $r_2= n+1-r$, and assume as we may that $r\leq n+1-r$. We have:
\[ K_{\mathcal F} \in \left( \begin{array}{c} (n+1)+w_1+w_2+1\\-9\end{array} \right).\]
\begin{lem}\label{weights}
The weighted blowup $F\colon \mathcal F\to \mathbb{P}$ restricts to a divisorial contraction $f\colon Z\to X$ with discrepancy $1$ if the weights $r, w_1, w_2$ of $F$ in \eqref{eq3} are one of:
\begin{itemize}
\item[1.] $i=0$ ($n=2$ or $3$) and $w_1=w_2=1$, $r=1$ or $2$, 
\item[2.] $w_1=w_2=2$ and $r\geq 2$ ($n\geq 3$).
\end{itemize}
\end{lem}
\begin{proof}
The $3$-fold $Z$ is a general complete intersection of $3$ hypersurfaces of degrees determined by $r, w_1$ and $w_2$. Once these degrees are known, we use adjunction to write the anticanonical class of $Z$. Since ${-}K_X\in \mathcal O(1)$ and $a=1$, $-K_Z\sim L-E$ and this yields the possible values for $r, w_1$ and $w_2$. 
For example, if $w_1, w_2, r\geq 2$, $Z$ is the proper transform of $X$ under $F$ and it is a complete intersection:
\[ \begin{cases}
u^{r-2}\alpha= x_0(u^{w_1-2}x_1+u^{w_2-2}x_2)-(x_3^2+x_4^2)\\
u^{n-1-r}\beta= x_0(u^{w_1-2}x_1-u^{w_2-2}x_2)-(x_3^2-x_4^2)\\
\alpha\beta+ u^{iw_1+ j-2} x_1^i x_3^j+u^{iw_2+ j-2} x_2^i x_4^j+ u^{4w_1-n-1}x_1^4+ u^{4w_2-n-1} x_2^4=0 \end{cases},\]
so that 
\[ Z= \left( \begin{array}{c}-2\\2\end{array}\right)\cap \left( \begin{array}{c}-2\\2\end{array}\right)\cap \left( \begin{array}{c}-(n+1)\\4\end{array}\right) \mbox{ and } -K_Z\in \left( \begin{array}{c}3-(w_1+w_2)\\1\end{array}\right),\]
and this forces $w_1=w_2=2$. Other cases are entirely similar. 
\end{proof} 

We check that, with one exception labelled ``bad link", for all weights in Lemma~\ref{weights}, the $2$-ray configuration on $\mathcal F$ induces a $2$-ray configuration on $TV(I,A)$ that induces a Sarkisov link for $X^{i,j}$. In the case of the bad link, the second birational contraction $g\colon Z\dashrightarrow Y$ has relatively trivial canonical class, so that $Y$ is not terminal and the $2$-ray configuration is not a Sarkisov link. Table 2 gives details of the construction of each Sarkisov link. 

\begin{thm}\label{main} There are examples of non-rigid terminal factorial quartic hypersurfaces in $\mathbb{P}^4$ with an isolated $cA_n$ point for all $2\leq n\leq 6$. 
For each combination $\big((i,j), n\big)$ in Table 1, $X^{i,j}\subset \mathbb{P}^4$ is a non-rigid terminal factorial quartic $3$-fold with an isolated singular point $(P\in X)$ of type $cA_n$. Table 2 lists Sarkisov links initiated by a divisorial contraction $f\colon Z\to X^{i,j}$ with centre at $(P\in X)$ and discrepancy $1$. Each entry specifies the weights of $(\alpha, \beta, x_3,x_4)$ for the germ of $f$ in the notation of Theorem~\ref{cAn}, and gives the explicit construction of the link. 
\begin{table}[ht]
\label{table2}
\scriptsize
\[\begin{array}{c||c|c|c}
p\in X&\text{blowup weights of $f$}&\text{decomposition of } \Phi_{\vert Z}&\text{Y/T}\\
\hline
\hline
cA_2&(2,1,1,1)&12\text{ flops}&Y_{3,4}\subset\mathbb{P}(1,1,1,1,2,2)\\
\hline
cA_3&(3,1,1,1)\text{ for }X^{1,2}& 8\text{ flops}& X^{1,2}\quad(\star)\\
\hline
&(3,1,1,1)\text{ for }X^{0,4}&&\text{ bad link }\\
\hline
&(2,2,1,1)&\text{4 flops}&dP_2\text{ fibration over }\mathbb{P}^1\\
\hline
cA_4&(3,2,1,1)&\text{2 flops then flip }(3,1,1,-1,-1;2)&dP_3\text{ fibration over }\mathbb{P}^1\\
\hline
cA_5&(4,2,1,1)&\text{2 flops}&\text{conic bundle over }\mathbb{P}(1,1,2)\\
\hline
&(3,3,1,1)&\mathbb Cong \text{then 2 flips }(3,1,1,-1,-1;2)&dP_4\text{ fibration over }\mathbb{P}^1\\
\hline
cA_6&(5,2,1,1)&\text{2 flops}&Y_{6,6}\subset\mathbb{P}(1,1,2,3,3,5)\\
\hline
&(4,3,1,1)&\mathbb Cong \text{then flip }(3,1,1,-1,-1;2)&\text{conic bundle over }\mathbb{P}(1,1,2)\\
\hline
\end{array}\]
\caption{}
\end{table}
\end{thm}
\begin{proof}
Each case is treated individually. To illustrate the computations involved, we treat the $cA_6$ case in detail. 
We then say a few words about the $cA_2$ case, where we recover the example of a non-rigid quartic constructed in \cite{CM04}. 

\begin{rem} The case labelled as $(\star)$ in Table 2 is a quadratic involution in the language of \cite{CPR}. In particular, the link does not produce a new Mori fibre space; it is just a birational selfmap of $X^{1,2}$ that is not an isomorphism.
\end{rem}
 
\noindent{\bf Non-rigid quartic with a $cA_6$ singular point.} Consider the terminal, factorial quartic hypersurface:
\[X= \{ (x_0x_1-x_3^2)^2 - (x_0x_2- x_4^2)^2+ x_1^3x_3+x_2^3x_4+x_1^4+x_2^4=0\}\subset \mathbb{P}^4.\]
As above, we embed $X$ as a complete intersection in $ \mathbb{P}= \mathbb{P}(1^5, 2^2)$, where the variables of weight $2$ are $\alpha, \beta$:
\[X= \begin{cases}
\alpha+ \beta= x_0x_1- x_3^2\\
\alpha - \beta= x_0x_2-x_4^2\\
\alpha \beta + x_1^3x_3+x_2^3x_4+x_1^4+x_2^4=0.
\end{cases}\subset \mathbb{P}(1^5, 2^2)
\] 
We construct Sarkisov links initiated by a divisorial contraction $f\colon Z\to X$, which is the restriction of a weighted blowup $F\colon \mathcal F\to \mathbb{P}$. 
We assume that the weights assigned to the variables $(\alpha, \beta, x_3, x_4)$ are those in Theorem~\ref{cAn}. By Theorem~\ref{cAn} and Lemma~\ref{weights}, $\mathcal F$ is the Picard rank $2$ toric variety $TV(I,A)$, where $A$ is of the form 
\[ \left(\begin{array}{cccccccc}
u&x_0&x_1&x_2&x_3&x_4&\alpha &\beta\\
1& 0&-2&-2&-1&-1&-r&-(7-r)\\
0&1&1&1&1&1&2&2
\end{array}\right), \mbox{ for } r=2 \mbox{ or } 3,\]
and $Z$ is a complete intersection of the form $(2L-2E, 2L-2E, 4L-7E)$ in $\mathcal F$. 
 
\begin{case} The germ of $f$ is a blowup with weights $(5,2,1,1)$. \end{case}
We re-order the coordinates of $\mathbb{P}$ and write the action $A$ as follows:
 \[
\left(\begin{array}{cccccccc}
u&x_0&x_3&x_4&\beta &x_1&x_2&\alpha\\
1& 0&-1&-1&-2&-2&-2&-5\\
0&1&1&1&2&1&1&2
\end{array}\right).\]
In this case, $Z$ is the complete intersection:
\[ Z= \begin{cases} u^3\alpha+ \beta= x_0x_1-x_3^2\\u^3\alpha-\beta= x_0x_2-x_4^2\\\alpha\beta+x_1^3x_3+x_2^3x_4+ x_1^4+x_2^4=0
 \end{cases}\subset \mathcal F.\] Taking the difference of the first two equations shows that the variable $\beta$ is redundant and:
 \[ Z=\begin{cases} u^3\alpha= x_0(x_1+x_2)-(x_3^2+x_4^2)\\
 \alpha(x_0(x_1-x_2)-(x_3^2-x_4^2))+x_1^3x_3+x_2^3x_4+ u(x_1^4+x_2^4)=0\end{cases} \subset \mathcal F,\] where we now denote by $\mathcal F$ the toric variety $TV(I,A)$, for 
 \[ A=\left(\begin{array}{ccccccc}
u&x_0&x_3&x_4 &x_1&x_2&\alpha\\
1& 0&-1&-1&-2&-2&-5\\
0&1&1&1&1&1&2
\end{array}\right).\]

The $2$-ray configuration on $\mathcal F$ is:
\[ \xymatrix{ 
& \mathcal F\ar[dl]_{F} \ar@{-->}[rr]^{\Phi} && \mathcal F^+ \ar[dr]^{G}&\\ 
\mathbb{P} &&&&\mathbb{P}^+ & }\]
where $\Phi$ is a small map that is the ample model for $L-(1+\varepsilon)E$ and $G\colon \mathcal F^+\to \mathbb{P}^+$ is a divisorial contraction, where $\mathbb{P}^+\simeq \Proj (\mathcal F, n(L-2E))$ for suitable $n>\!\!>1$. 

The only pure monomials in $u, x_0,x_3, x_4$ in the equation of $Z$ are in the expression $x_3^2+ x_4^2$ in the first equation, so that the restriction of $\Phi_1$ to $Z$ is a flop in $2$ lines (a copy of $\mathbb{P}^1_{u,x_0}$ above each of the two points $\{x_3^2+x_4^2=0\}\subset \mathbb{P}^1_{x_3,x_4}$), and since $\alpha$ doesn't divide the equations of $Z$, $G$ does restrict to a divisorial contraction. 

To determine $\mathbb{P}^+$, we find a suitable change of basis in which to express the action $A$. In practice, we look for a matrix $M$ in $\mbox{Sl}_2(\mathbb Z)$ such that \[M\cdot\left( \begin{array}{c}-5\\2\end{array}\right)= \left(\begin{array}{c}\ast \\0\end{array}\right).\] The matrix $M= \left( \begin{array}{cc}1 &3\\-2&-5\end{array}\right)$ transforms the action $A$ into:
\[ 
  \left(\begin{array}{ccccccc}
u&x_0&x_3&x_4 &x_1&x_2&\alpha\\
1& 3&2&2&1&1&1\\
-2&-5&-3&-3&-1&-1&0
\end{array}\right),\]
so that $\mathbb{P}^+=\mathbb{P}(1^2, 2, 3^2, 5)$, with coordinates $x_1,x_2$ (degree $1$), $u\alpha$ (degree $2$), $x_3\alpha, x_4\alpha$ (degree $3$) and $x_0\alpha^2$ (degree $5$). 
Writing the equations of the proper transform of $Z^+$ shows that $Y\subset \mathbb{P}^+$ is 
\[ \begin{cases}
(u\alpha)^3= (x_0\alpha^2)(x_1+x_2)-((x_3\alpha)^2+ (x_4\alpha)^2)\\
(x_0\alpha^2)(x_1-x_2) - ((x_3\alpha)^2-(x_4\alpha)^2) +x_1^3(x_3\alpha) + x_2^3(x_4\alpha)+ (u\alpha)(x_1^4+x_2^4)=0\end{cases}\]
so that $Y$ is the complete intersection of two hypersurfaces of degree $6$ in $\mathbb{P}^+$. Since $x_3\alpha$ is a section of $g^*\OO_{\mathbb{P}^+}(3)$, the $3$-fold $Y$ has Fano index $3$, and by construction of $f$, its basket consists of a single $[5,2]$ singular point at $P_{\alpha}$; $Y$ is the Fano variety number 41920 in \cite{GRDB}. 

\begin{case}The germ of $f$ is a blowup with weights $(4,3,1,1)$. \end{case}
We re-order the coordinates of $\mathbb{P}$, and write the action $A$:
 \[
\left(\begin{array}{cccccccc}
u&x_0&x_3&x_4&\beta &x_1&x_2&\alpha\\
1& 0&-1&-1&-3&-2&-2&-4\\
0&1&1&1&2&1&1&2
\end{array}\right).\]The $3$-fold $Z$ is the complete intersection:
\[ Z= \begin{cases} u^2\alpha= x_0(x_1+x_2)-(x_3^2+x_4^2)\\u \beta= x_0(x_1-x_2)-(x_3^2-x_4^2)\\\alpha\beta+ x_1^3x_3+x_2^3x_4+u(x_1^4+x_2^4)=0
 \end{cases}\subset \mathcal F.\]

The $2$-ray configuration on $\mathcal F$ is:
\[ \xymatrix{ 
& \mathcal F\ar[dl]_{F} \ar@{-->}[rr]^{\Phi_1} && \mathcal F_1 \ar@{-->}[rr]^{\Phi_2}&&\mathcal F_2 \ar[dr]^{G}&\\ 
\mathbb{P} &&&&&&\mathbb{P}^+ & }\]
where $\Phi_1$ is a small map and $\mathcal F_1$ the ample model of $L-(1+ \varepsilon)E$, $\Phi_2$ is a small map and $\mathcal F_2$ the ample model of $2L-(3+ \varepsilon)E$, and $G$ is a divisorial contraction to $\mathbb{P}^+\simeq \Proj(\mathcal F, n(L-2E))$ for suitable $n>\!\!>1$. 

Since $x_3^2, x_4^2$ appear in two of the equations defining $Z$, $Z$ does not contain any curve contracted by $\Phi_1$ and ${\Phi_1}_{\vert Z}$ is an isomorphism. We still denote by $Z$ its image under $\Phi_1$.

We show that ${\Phi_2}_{\vert Z}$ is a flip. We study the behaviour of $Z$ near $P_{\beta}=(0{:}0{:}0{:}0{:}1{:}0{:}0{:}0)$. The restriction of  $Z$ to $U_{\beta}=\{\beta=1\}$ is a hypersurface: we may use the second and third equation to eliminate $u$ and $\alpha$, so that $Z\cap U_{\beta}$ is the hypersurface defined by the first equation. As above, under the change of coordinates associated to \[\left(\begin{array}{cc} 1 &1\\2&3\end{array}\right)\in \mbox{SL}_2(\mathbb Z),\] the action becomes:
\[ 
  \left(\begin{array}{cccccccc}
u&x_0&x_3&x_4&\beta &x_1&x_2&\alpha\\
1& 1&0&0&-1&-1&-1&-2\\
2&3&1&1&0&-1&-1&-2
\end{array}\right),\]
so that once the variables $u, \alpha$ and the second and third equations defining $Z$ have been eliminated, we are left with a flip of the hypersurface defined by the first equation of $Z$ in $\mathbb C_{x_0,x_3,x_4, x_1,x_2}$, which is \[x_0(x_1+x_2)+x_3^2+x_4^2+\cdots=0,\] that is, in the notation of \cite{Br99}, of the form $(3,1,1,-1,-1; 2)$. There are thus $2$ flipped curves, and while $Z$ has a $cA/3$ singularity over $P_{\beta}$, $Z_2$ is Gorenstein over $P_{\beta}$. The map $G$ is a fibration over $\mathbb{P}(\alpha, x_1,x_2)= \mathbb{P}(1,1,2)$, and the equations of $Z$ show that the restriction $Z_2\to \mathbb{P}(1,1,2)$ is a conic bundle.  
 
\noindent{\bf Non-rigid quartic with a $cA_2$ point.} Consider the terminal, factorial quartic hypersurface
\[ 
X=\{ (x_0x_1-x_3^2)^2 - (x_0x_2- x_4^2)^2+ x_0(x_3^3+x_4^3)+ x_1^4+x_2^4=0\}\subset \mathbb{P}^4.\]

By Lemma~\ref{weights},  after re-ordering the coordinates of $\mathbb{P}$, $\mathcal F$ can only be the Picard rank $2$ toric variety $TV(I,A)$, where 
\[ 
  \left(\begin{array}{cccccccc}
u&x_0&\alpha &x_1&x_2&x_3 &x_4&\beta\\
1& 0&-1&-1&-1&-1&-1&-2\\
0&1&2&1&1&1&1&2
\end{array}\right),\]
and $Z$ is given by the equations:
\[ \begin{cases}
\alpha= x_0(x_1+x_2)- u(x_3^2+x_4^2)\\
u\beta= x_0(x_1-x_2)- u(x_3^2-x_4^2)\\
\alpha\beta + x_0(x_3^3+x_4^3)+ u(x_1^4+x_2^4)=0.
\end{cases}\]
The first equation shows that we may eliminate the variable $\alpha$. The first step of the $2$-ray configuration of $\mathcal F$ is a small map but introduces a new divisor on $Z$, hence is not a step in the $2$-ray game of $Z$. Note that the second equation is in the ideal $(u,x_0)$, we will re-embed $Z$ by unprojection into a toric variety of Picard rank $2$ whose $2$-ray configuration restricts to suitable maps on $Z$. To do so, we introduce an unprojection variable $s$ and replace the second equation with:
\[ \begin{cases}  
sx_0 = \beta + x_3^2-x_4^2\\
su= x_1-x_2
\end{cases}\] 
From the two equations that define $s$, we eliminate the variable $\beta$ and $x_1$ so that the (isomorphic) image of $Z$ under the unprojection is:
\begin{eqnarray*}
\{\big( x_0(2x_1-su)-u(x_3^2+x_4^2)\big)(sx_0-x_3^2+x_4^2) + x_0(x_3^3+x_4^3)+ u(x_1^4+(x_1-su)^4)=0\}\\
 \subset  \left(\begin{array}{ccccccc}
u&x_0 &x_1&x_3 &x_4& s\\
1& 0&-1&-1&-1&-2\\
0&1&1&1&1&1
\end{array}\right),\end{eqnarray*}
Again, we see that the equation defining $Z$ is in the ideal $(u,x_0)$. If we denote by $f,g$ (non-unique) polynomials such that $Z=\{ x_0 f+ u g=0\}$, so that we have \[ f= x_3^3+ x_4^3+ \cdots \mbox{ and } g= x_1^4+ (us)^4+ \cdots,\] and introduce a second unprojection variable 
\[ t= \frac{f}{u}= \frac{g}{x_0}\in \left(\begin{array}{c} -4\\3 \end{array}\right),\]
then we see that $Z$ is the complete intersection
\[ \begin{cases}
tu= f\\
tx_0= g\end{cases}\subset \mathcal F= \left(\begin{array}{ccccccc}
u&x_0 &x_1&x_3 &x_4&t & s\\
1& 0& -1&-1&-1&-4& -2\\
0&1&1&1&1&3&1
\end{array}\right).\]
The $2$-ray configuration of $\mathcal F$ is 
\[ \xymatrix{ 
& \mathcal F\ar[dl]_{F} \ar@{-->}[rr]^{\Phi} && \mathcal F^+ \ar[dr]^{G}&\\ 
\mathbb{P} &&&&\mathbb{P}^+ & }\]
where $\Phi$ is a small map and $\mathcal F^+$ the ample model of $L-(1+\varepsilon)E$ and $G$ is a divisorial contraction. The restriction $\Phi_{\vert Z}\colon Z \dashrightarrow Z^+$ is a flop in $12$ lines that are copies of $\mathbb{P}^1_{u,x_0}$ lying over the $12$ points \[\{ x_1^4+x_3^4-x_4^4=x_2^3+x_3^3+\cdots=0\}\subset \mathbb{P}_{x_1,x_3,x_4}.\] Since $Z \not \subset \{ s=0\}$, the restriction $G_{\vert Z^+} \colon Z^+\to Y$ is a divisorial contraction. As above, applying the coordinate change for the action associated to \[\left(\begin{array}{cc} 1 &1\\1&2\end{array}\right)\in \mbox{SL}_2(\mathbb Z)\] transforms $A$ into 
\[ 
  \left(\begin{array}{ccccccc}
u&x_0&x_1&x_3&x_4 &t&s\\
1&1 &0&0&0&-1&-1\\
1&2&1&1&1&2&0
\end{array}\right),\]
and we see that $\mathbb{P}^+= \mathbb{P}(1^4, 2^2)$, with coordinates $su, x_1, x_3, x_4$ (degree $1$), $sx_0, st$ (degree $2$). The proper transform of $Z$ is
\[ Y =\begin{cases} x_1^4+ \cdots =0\\ x_2^3+x_3^3+ \cdots=0 \end{cases}\subset \mathbb{P}^+,\]
i.e.~$Y$ is the complete intersection of a cubic and a quartic hypersurface in $\mathbb{P}^+$, a Fano $3$-fold of codimension $2$ and genus $2$. The map $Z^+ \to Y$ is a Kawamata blowup of one of the two $1/2(1,1,1)$ points on $Y$. This link is constructed in \cite{CM04}.
\end{proof}
\begin{rem} In several cases, one or both of the variables $\alpha, \beta$ are redundant. This means that $f\colon Z\to X$ is the restriction of a weighted blowup of some $\mathbb{P}^\prime$ with $\mathbb{P}^4\subset \mathbb{P}^\prime\subset \mathbb{P}$: the construction could have been obtained with a ``smaller embedding". For example, this is the case in our treatment of a terminal factorial quartic hypersurface with a $cA_2$ point: \cite{CM04} construct the same link without introducing $\alpha, \beta$. In a given example, it is usually clear how many (if any) variables need to be introduced. We have chosen to always introduce two variables (and then eliminate redundant ones) in order  to present our results in a unified way.
\end{rem}
\begin{rem} It is crucial to understand that we make no claim about the existence of Sarkisov links initiated by divisorial contractions $f\colon Z\to X$ whose germs have weights different from those in Lemma~\ref{weights}, where $X$ is one of the hypersurfaces $X^{i,j}$. Such divisorial contractions may occur, but they are not restrictions of weighted blowups of $\mathbb{P}=\mathbb{P}(1^5, 2^2)$. We expect that in some cases, one may construct such contractions by considering a different embedding of $X\subset \mathbb{P}^\prime$ and looking at restrictions of weighted blowups of $\mathbb{P}^\prime$.  
For instance, we do not know wether $X^{3,1}$ admits a Sarkisov link initiated by a divisorial contraction whose germ is a weighted blowup $(6,1,1,1)$.  
\end{rem}

\subsection{Other examples with $cA$ singularities}\label{otherex}
In this section, we use similar techniques to give an example of a non-rigid terminal factorial quartic hypersurface with a $cA_7$ singular point. These can also be used to construct Sarkisov links initiated by divisorial contractions with discrepancy $a>1$ and centre at a $cA_n$ point for $n\geq 2$, where the possible values $a,n$ are determined in Proposition~\ref{pro1}. We give an example with $n=2$ and $a=2$.  
\begin{exa}\label{cA7}
Let $X$ be the hypersurface: 
\[ X= \{ (x_0x_1-x_3^2)^2-(x_0x_2-x_4^2)^2+ x_0x_1^3-x_1^2x_3^2+ x_1^4+x_2^4=0\} \subset \mathbb{P}^4.\]
The restriction of $X$ to the affine chart $U_0=\{ x_0=1\}$ is 
\[ F(x,y,z,t)= (x-z^2)^2-(y-t^2)^2+ x^3-x^2z^2+x^4+y^4\]
and the singularity $(F=0)$ is equivalent to $(G=0)$, where $G$ defines a $cA_7$ singularity. To see this, let $\underline \varphi = (\varphi_1, \varphi_2, \varphi_3, \varphi_4)\in \Aut \mathbb C[[x,y,z,t]]$, where 
\[ \begin{cases}
\varphi_1(x,y,z,t)= x+y-(z^2(1-\frac{1}{2}z^2)+t^2)\\
\varphi_2(x,y,z,t)= x-y-(z^2(1-\frac{1}{2}z^2)-t^2)\\
\varphi_3(x,y,z,t)= z,\\ \varphi_4(x,y,z,t)=t\end{cases}
\]
and check that $F(x,y,z,t)= G(\varphi_1, \varphi_2, \varphi_3, \varphi_4)$, where 
\begin{equation}\label{eq2'} G(x,y,z,t)= xy+ z^{8}+ t^{8}+ G'(x,y,z,t),\end{equation}
where $G'(x,y,z,t)$ is a polynomial of degree strictly greater than $1$ with respect to the weights \[\underline w=(1/2, 1/2, 1/8,1/8).\]
Since $G$ is a semi quasi-homogeneous polynomial of degree $1$ with respect to $\underline w$, and since no element of its Jacobian algebra has degree strictly greater than $1$ with respect to these weights, \[G\sim xy+ z^{8}+ t^{8}\] and $(P\in X)$ is a $cA_7$ point.  
The hypersurface $X\subset \mathbb{P}^4$ is a terminal and factorial quartic hypersurface, and hence is a Mori fibre space.

As in Section~\ref{family}, we embed $X$ as the complete intersection
\[ \begin{cases} 
\alpha= x_0(x_1+x_2) - (x_3^2+x_4^2)\\
\beta = x_0(x_1-x_2)-  (x_3^2-x_4^2)\\
\alpha \beta + 2 x_1^2(\alpha+ \beta)+ x_1^4+ x_2^4=0\end{cases}\subset \mathbb{P}= \mathbb{P}(1^4, 2^2).\]
We consider the weighted blowup $F\colon \mathcal F\to \mathbb{P}$, where $\mathcal F$ is the Picard rank $2$ toric variety $TV(I,A)$, where $I= (u,x_0)\cap (x_1,\cdots , x_4, \alpha, \beta)$ is the irrelevant ideal and $A$ is the action of $\mathbb C^*\times \mathbb C^*$ with weights:
\[ \left(\begin{array}{cccccccc}
u&x_0&x_3&x_4&x_1&x_2&\alpha &\beta\\
1& 0&-1&-1&-2&-2&-4&-4\\
0&1&1&1&1&1&2&2
\end{array}\right) \]

The proper transform of $X$ is 
\[ Z= \begin{cases} u^2\alpha= x_0(x_1+x_2) - (x_3^2+x_4^2)\\
u^2\beta = x_0(x_1-x_2)-  (x_3^2-x_4^2)\\
\alpha \beta + 2 x_1^2(\alpha+ \beta)+ x_1^4+ x_2^4=0\end{cases}= \left(\begin{array}{c}2\\2\end{array}\right) \cap \left(\begin{array}{c}2\\2\end{array}\right)\cap \left(\begin{array}{c}8\\4\end{array}\right)\subset \mathcal F\]

and we check that the restriction of $F$ to $Z$ is indeed a divisorial contraction $f\colon Z\to X$ with discrepancy $a=1$.  
The $2$-ray configuration on $\mathcal F$ is:
\[ \xymatrix{ 
& \mathcal F\ar[dl]_{F} \ar@{-->}[rr]^{\Phi} && \mathcal F^+  \ar[dr]^{G}&\\ 
\mathbb{P} &&&&\mathbb{P}^+  } \]
where $\Phi$ is a small map and $\mathcal F^+$ is the ample model of $L-(1+ \varepsilon)E$ and $G$ is a fibration morphism and $\mathbb{P}^+= \mathbb{P}(1,1,2,2)= \mathbb{P}_{x_1,x_2, \alpha, \beta}$ is the ample model $\Proj(n(L-2E))$ for suitable $n>\!\!>1$. 

The restriction of $\Phi$ to $Z$ is an isomorphism because the monomials $x_3^2, x_4^2$ appear in the first two equations, we still denote by $Z$ its image. The restriction of $G$ to $Z$ is a conic bundle over the quartic surface $S_4\subset \mathbb{P}(1,1,2,2)$ defined by the third equation of $Z$. 
\end{exa}

The results of Section~\ref{family} and Example~\ref{cA7} thus show:

\begin{thm}
There are examples of non-rigid terminal factorial quartic hypersurfaces with a singular point of type $cA_n$ for all possible $n\geq 2$. 
\end{thm}

Each of the Sarkisov links we have constructed so far is initiated by a divisorial contraction with centre at a $cA_n$ point and discrepancy $a=1$. We now construct an example with higher discrepancy. 
\begin{exa}\label{antiflip}
We construct a Sarkisov link initiated by a divisorial contraction with centre at a $cA_2$ point. Unlike the previous examples, which only involved flips and flops, the Sarkisov link in this example involves an antiflip.  

Consider the terminal, factorial quartic hypersurface
\[ X=\{ x_0x_1(x_0x_2-x_4^2) + x_0(x_2^3+x_3^3)- x_1^4+ x_2^4+ x_3^4\}\subset \mathbb{P}^4. \]
Then $\Sing X=\{ P, P_4\}$, and, setting $x_0=1$ in the expression above, we see that:
\[ (P\in X) \sim 0\in \{ xy+z^3+t^6+ (\mbox{higher order terms})=0\},\]
so that $P$ is a $cA_2$ point and a divisorial contraction with centre at $(P\in X)$ has discrepancy $1$ or $2$ by Theorem~\ref{cAn}. We embed $X$ as
\[ X= \begin{cases}\alpha= x_0x_1\\\beta = x_0x_2-x_4^2\\\alpha \beta x_0(x_2^3+x_3^3)- x_1^4+ x_2^4+ x_3^4=0\end{cases}\subset \mathbb{P}(1^5, 2^2).\] 

If we consider a divisorial contraction with discrepancy $1$, we obtain a link of the same form as above. We now consider the case when $f$ is a divisorial contraction with discrepancy $2$. Then, by Theorem~\ref{cAn}, the weights of $(\alpha, \beta, x_3,x_4)$ are either $(1,5,2,1)$ or $(3,3,2,1)$. We consider the second case, and as in Lemma~\ref{weights}, we show that $\mathcal F$ is a toric variety $TV(I,A)$, where 
\[ A= \left(\begin{array}{cccccccc}
u&x_0 &x_1&x_2&x_3 &x_4&\alpha &\beta \\
1&0 &-p&-q&-2&-1&-3&-3\\
0&1&1&1&1&1&2&2
\end{array}\right), \mbox{ for } (p,q)= (2,2), (2,3) \mbox{ or } (3,2).
\]
We consider the case when $(p,q)=(2,2)$, so that, after re-ordering, 
\[ A= \left(\begin{array}{cccccccc}
u&x_0 &x_4&\alpha & \beta & x_1 &x_2&x_3  \\
1&0 &-1&-3&-3&-2&-2&-2\\
0&1&1&2&2&1&1&1
\end{array}\right),\]
and the proper transform of $X$ is given by the equations:
\[ 
\begin{cases} u\alpha = x_0x_1\\ u\beta = x_0x_2-x_4^2\\\alpha \beta x_0(x_2^3+x_3^3)+ u^2(-x_1^4+ x_2^4+ x_4^2)=0\end{cases}.\]
As in the proof of the $cA_2$ case in Theorem~\ref{main}, the first equation is in the ideal $(u,x_0)$, we re-embed $Z$ so that it follows the ambient $2$-ray configuration: we introduce an unprojection variable $s$ such that:
\[ \begin{cases} su=x_1\\ sx_0= \alpha \end{cases}, \quad s\in \left(\begin{array}{c}-3\\1\end{array}\right).\]
We then see that the variables $x_1$ and $\alpha$ are redundant so that the expression of $Z$ is 
\[ 
\begin{cases} 
u\beta = x_0x_2-x_4^2\\
x_0(s\beta+ x_3^3+x_2^3)+ u^2((su)^4+ x_2^4+x_3^4)= 0.
\end{cases}
\]
As above, since the second equation is in the ideal $(u^2, x_0)$, we need to introduce a second unprojection variable $\eta$ such that 
\[ 
\begin{cases} 
\eta u^2= s\beta+ x_3^3+x_2^3\\
\eta x_0= (su)^4+ x_2^4+x_3^4
\end{cases} , \quad \eta \in \left(\begin{array}{c}-8\\3\end{array}\right).\]
We now get that $Z$ is the complete intersection 
\[ 
\begin{cases} 
u\beta = x_0x_2-x_4^2\\\eta u^2= s\beta+ x_3^3+x_2^3\\
\eta x_0= (su)^4+ x_2^4+x_3^4
\end{cases} \subset \mathcal F,\]
where $\mathcal F$ denotes the toric variety $TV(I,A)$, for 
\[ A= \left(\begin{array}{cccccccc}
u&x_0 &x_4 & \beta & x_2 &x_3&\eta &s  \\
1&0 &-1&-3&-2&-2&-8&-3\\
0&1&1&2&1&1&3&1
\end{array}\right).\]

The $2$-ray configuration on $\mathcal F$ is:
\[ \xymatrix{ 
& \mathcal F\ar[dl]_{F} \ar@{-->}[rr]^{\Phi_1} && \mathcal F_1 \ar@{-->}[rr]^{\Phi_2}&&\mathcal F_2 \ar@{-->}[rr]^{\Phi_3}&&\mathcal F_3 \ar[dr]^{G}&\\ 
\mathbb{P} &&&&&&&&\mathbb{P}^+ & }, \]
\begin{itemize}
\item[-] $\Phi_1\colon \mathcal F \dashrightarrow \mathcal F_1$ is a small map, and $\mathcal F_1$ is the ample model of $L-(1+ \varepsilon)E$,
\item[-] $\Phi_2\colon \mathcal F_1 \dashrightarrow \mathcal F_2$ is a small map, and $\mathcal F_2$ is the ample model of $2L-(3+ \varepsilon)E$, 
\item[-] $\Phi_3\colon \mathcal F_2 \dashrightarrow \mathcal F_3$ is a small map, where $\mathcal F_3$ is the ample model of $L-(2+ \varepsilon)E$,
\item[-] $G\colon \mathcal F_3\to \mathbb{P}^+$ is a divisorial contraction, and $\mathbb{P}^+$ is the ample model of $8L-3E$. 
\end{itemize}

We study the restriction of this $2$-ray configuration to $Z$. 
Since the monomial $x_4^2$ appears in one of the equations of $Z$, the restriction ${\Phi_1}_{\vert Z}$ is an isomorphism. We still denote by $Z$ its image. 

We now prove that ${\Phi_2}_{|Z}\colon Z\dashrightarrow Z_2$ is a small birational map. Since it is a $K$-positive contraction, we also need to prove that $Z_2$ has terminal singularities. 
The exceptional locus of ${\Phi_2}_{\vert Z}$ is at most $1$-dimensional, as the only pure monomial in $u, x_0, x_4$ that appears in the equations of $Z$ is $x_4^2$.The exceptional locus of $({\Phi_2}_{|Z})^{-1}$ is also at most $1$-dimensional, as pure monomials in $x_2, x_3, \eta , s$ appear in two of the three equations defining $Z$. 
 
In order to study ${\Phi_2}_{|Z}$, we localise near $P_{\beta}$. Setting $\beta=1$, we use the first two equations to eliminate the variables $u$ and $s$, so that  $X\cap \{\beta=1\}$ is the hypersurface defined by the third equation:
\[ \{\eta x_0= x_2^4+ x_3^2+ (\eta u^2-x_3^3-x_2^3)^4(x_0x_2-x_4^2)^4=0\}\]

As above, under the change of coordinates associated to \[\left(\begin{array}{cc} 1 &-1\\-2&3\end{array}\right)\in \mbox{SL}_2(\mathbb Z),\] the action becomes:
\[ 
  \left(\begin{array}{cccccccc}
u &x_0 &x_4 & \beta & x_2 &x_3&\eta &s  \\
-1&-1 &0&1&1&1&5&2\\
2&3&1&0&-1&-1&-7&-3
\end{array}\right),\]
so that once the variables $u, s$ and the second and third equations defining $Z$ have been eliminated, we are left with the inverse of a flip of the hypersurface defined by the first equation of $Z$ in $\mathbb C_{\eta, x_2, x_3, x_0,x_4}$, which is, in the notation of \cite{Br99}, of the form $(7,1,1,-3,-1; 4)$. The map ${\Phi_2}_{\vert Z}$ is thus an antiflip between $3$-folds with terminal singularities. The exceptional locus of ${\Phi_2}_{\vert Z}$ is empty because the equation of $X\cap U_{\beta}$ has no pure monomial in $x_0, x_4$; the exceptional locus of $({\Phi_2}_{\vert Z})^{-1}$ consists of $4$ lines  \[\{ x_2^4+ x_3^4=0\} \subset \mathbb{P}(1,1,7)=\mathbb{P}_{x_2, x_3, \eta}.\]   
By the construction of $f$, the basket of $Z$ consists of two $[3,1]$ singularities, one of which lies over $P_{\beta}$. From \cite{Br99}, the basket of $Z_2$ consists of one $[7,1]$ singular point and one $[3,1]$ singular point.

The restriction ${\Phi_3}_{\vert Z_2}$ is an isomorphism because $x_3^3+x_2^3$ appears in the second equation defining $Z$ and $x_3^4+x_2^4$ in the third. 
We still denote $Z_2$ its image by ${\Phi_3}_{\vert Z_2}$.

The restriction of $G$ to $Z_2$ is a morphism because $Z_2\cap \{ s=0\}$ is a prime divisor. 
As above, under the change of coordinates associated to \[\left(\begin{array}{cc} 1 &-2\\-1&3\end{array}\right)\in \mbox{SL}_2(\mathbb Z),\] the action becomes:
\[ 
  \left(\begin{array}{cccccccc}
u &x_0 &x_4 & \beta & x_2 &x_3&\eta &s  \\
-1&-2 &-1&-1&0&0&2&1\\
1&3&2&3&1&1&1&0
\end{array}\right),\]
so that $\mathbb{P}^+= \mathbb{P}(1^4, 2, 3^2)$ with coordinates $\eta, s^3u,s^2x_2, s^2x_3$ (degree $1$), $s^5x_4$ (degree $2$) and $s^7\beta, s^8x_0$ (degree $3$).

The proper transform of $Z$ is 
\[ Y =\begin{cases}
(s^3u)(s^7\beta)= (s^8x_0)(s^2x_2)- (s^5x_4)^2\\
s^7\beta= (s^3u)^2\eta- (s^2x_3)^3-(s^2x_2)^3\\
(s^8x_0)\eta= (s^3u)^4+(s^2x_2)^4+(s^2x_3)^4
\end{cases}\subset \mathbb{P}^+,
\]
so that the variable $s^7\beta$ and the second equations are redundant. The $3$-fold $Y$ is a complete intersection of two quartic hypersurfaces $Y_{4,4}\subset \mathbb{P}(1^4,2,3)$. This is a Fano $3$-fold of index $1$, and is the Fano variety number 16204 in \cite{GRDB}. 

Setting $\eta=1$ shows that $Y_{4,4}\subset \mathbb{P}^+$ as a singular point of type $cE_7$ at $P_{\eta}$, and the contraction $g\colon Z_2\to Y$ is a divisorial contraction with centre at $P_{\eta}$. The contraction $g$ has discrepancy $2$, and the basket of $Z_2$ shows that $g$ is, in Kawakita's notation \cite{Kawk03}, a contraction of type I. 

In summary, we have constructed a Sarkisov link from a terminal factorial quartic hypersurface $X$ with a $cA_2$ point $(P\in X)$ to a complete intersection $Y_{4,4}\subset \mathbb{P}(1^4, 2, 3)$ with a $cE_7$ point $(Q\in Y)$, which is of the form:

\begin{minipage}[c]{4cm}
\[ \xymatrix{ Z\ar@{-->}[rr]^{\Phi} \ar[d]_f && Z^+\ar[d]^g\\
X&& Y_{4,4}}\]
\end{minipage}
\begin{minipage}{\textwidth}
\begin{itemize}
\item[-] $f$ is a discrepancy $2$ divisorial contraction with centre at $P$, 
\item[-] $\varphi= \varphi_3\circ\varphi_2\circ\varphi_1$, with $\varphi_1, \varphi_3$ isomorphisms and $\varphi_2$ an antiflip, 
\item[-] $g$ is a discrepancy $2$ divisorial contraction with centre at $Q$.
\end{itemize}
\end{minipage}
\end{exa}

\subsection{Examples with $cD$ and $cE$ singularities}\label{cD,cE}

We now give examples of non-rigid factorial quartic hypersurfaces with singular points that are not of type $cA$. 
The study of the pliability of quartics with $cD_m$ and $cE_{6,7,8}$ singular points is complicated by the fact that, unlike in the $cA_n$ case, there is no classification of the germs of divisorial extractions $f\colon Z \to X$ with centre at a $cD$ or $cE$ point.  We only know the germs of a few explicit divisorial extractions with these centers: those that are weighted blowups with discrepancy $1$. The following examples are non-rigid quartic hypersurfaces with a $cD$ or $cE$ singular point.

\begin{exa}\label{cD4}
Let \[X= \{ x_0^2x_1^2- x_0(x_2^3+ x_3^3+ x_4^3)+ f_4(x_1, x_2, x_3, x_4)= 0\}\subset \mathbb{P}^4\] be a quartic hypersurface, where $f_4$ is a general homogeneous polynomial of degree $4$ such that  
$ \{ x_2^3+ x_3^3+ x_4^3= f_4(0, x_2, x_3, x_4)=0\}\subset \mathbb{P}^2$
consists of $12$ points. By generality of $f_4$, $X$ is factorial and $\Sing X$ consists of a single point $(1{:}0{:}0{:}0{:}0)$ that is of type $cD_4$. The restriction of the weighted blowup $\mathcal F\to \mathbb{P}$ with centre at $P$ and weights $(2,1,1,1)$ is a divisorial contraction with discrepancy $1$. It initiates a Sarkisov link that consists of a flop in $12$ lines over the $12$ points \[\{ x_2^3+ x_3^3+ x_4^3= f_4(0, x_2, x_3, x_4)=0\}\subset \mathbb{P}^2_{x_2, x_3, x_4}\] followed by a divisorial contraction to a $1/2(1,1,1)$ point on a quasismooth intersection of a quartic and a cubic $Y= Y_{3,4}\subset \mathbb{P}(1^4, 2^2)$. Note that the Fano $3$-fold $Y$ is not a general $Y_{3,4}\subset \mathbb{P}(1^4, 2^2)$. Denoting by $x_1, \cdots, x_4$ the coordinates of weight $1$ and by $y_1, y_2$ those of weight $2$, the equation of $Y$ is of the form
\[ Y\colon \begin{cases} y_1x_1-( x_2^3+ x_3^3+x_4^3)=0\\ y_1y_2-y_1^2-f_4(x_1, \cdots, x_4)=0\end{cases} \subset \mathbb{P}(1^4, 2^2).\]

Note the similarity of this link to the link between a quartic $X^\prime\subset \mathbb{P}^4$ with a $cA_2$ point and a general $Y_{3,4}\subset \mathbb{P}(1^4, 2^2)$ studied in \cite{CM04}. In our case, $Y$ is a special quasi-smooth model in its family, and we conjecture that as in \cite{CM04} $X$ is birigid, i.e.~$\mathcal P(X)= \{ [X], [Y]\}$.  

\end{exa}
\begin{exa}\label{cD5} 
Let \[X=  \{ x_0^2x_1^2+ x_0x_2^2x_3+ x_1^4+ x_2^4+ x_3^4+ x_4^4= 0\}\subset \mathbb{P}^4;\] $\Sing X$ consists of a single point $(1{:}0{:}0{:}0{:}0)$ that is of type $cD_5$.
The restriction of the weighted blowup $\mathcal F\to \mathbb{P}$ with centre at $P$ and weights $(2,1,2,1)$ is a divisorial contraction with discrepancy $1$. It initiates a Sarkisov link that consists of a flop in $4$ lines lying above $\{ x_2^4+ x_4^4=0\}\subset \mathbb{P}^1_{x_2, x_4}$ followed by a del Pezzo fibration of degree $2$. 
\end{exa}

\begin{exa}\label{cE6}{\bf (Singular points of $cE$ type)}

In the same way, let \[X=  \{ x_0^2x_1^2+ x_0x_2^3+ x_1^4+ x_2^4+ x_3^4+ x_4^4= 0\}\subset \mathbb{P}^4;\] $\Sing X$ consists of a single point $(1{:}0{:}0{:}0{:}0)$ that is of type $cE_6$.
The restriction of the weighted blowup $\mathcal F\to \mathbb{P}$ with centre at $P$ and weights $(2,2,1,1)$ is a divisorial contraction with discrepancy $1$. It initiates a Sarkisov link that consists of a flop in $4$ lines lying above $\{ x_3^4+ x_4^4=0\}\subset \mathbb{P}^1_{x_3, x_4}$ followed by a del Pezzo fibration of degree $2$.

Similarly, if 
\begin{eqnarray*} Y= \{ (x_0x_1-x_4^2)^2+x_0x_2^3+x_2x_3^3+x_1^4+ x_2^4=0\} \subset \mathbb{P}^4 \mbox{ and }\\
 Z= \{ (x_0x_1-x_3^2-x_4^2)^2+x_0x_2^3+x_0x_1^2(x_3+x_4)+x_1^4+x_2^4=0\}\subset \mathbb{P}^4\end{eqnarray*}
then \[ \Sing Y= \Sing Z= \{ P\}= (1{:}0{:}0{:}0{:}0)\]
and $(P\in Y)$ is of local analytic type $cE_7$, while $(P\in Z)$ is of local analytic type $cE_8$. The $3$-folds $Y$ and $Z$ are non-rigid. 
Indeed, embed $Y$ and $Z$ as complete intersections 
\[ Y = \begin{cases} \alpha= x_0x_1-x_4^2\\
\alpha^2  + x_0x_2^3+x_2x_3^3+x_1^4+ x_2^4=0\end{cases} \subset \mathbb{P}_1=\mathbb{P}(1^5, 2)\]
where the variable of degree $2$ is $\alpha$, and  
\[ Z= \begin{cases}\beta= x_0x_1-x_3^2-x_4^2\\ \beta^2 +x_0x_2^3+x_0x_1^2(x_3+x_4)+x_1^4+x_2^4=0\end{cases} \subset \mathbb{P}_2= \mathbb{P}(1^5, 2)
\]
where the variable of degree $2$ is $\beta$. Then consider the rank $2$ toric varieties $\mathcal F_1$ and $\mathcal F_2$ associated to 
\begin{equation*}
\left( \begin {array}{ccccccc}
 u & x_0& x_1&x_2 & x_3&x_4& \alpha \\
1 & 0 & -2 & -2 &-2 &-1& -3\\
 0 & 1 & 1& 1& 1& 1 &2 \end{array}
\right) \mbox{ and } 
\left( \begin {array}{ccccccc}
 u & x_0& x_1&x_2 & x_3&x_4& \beta\\
1 & 0 & -3 & -2 &-1 &-1 & -3\\
 0 & 1 & 1& 1& 1& 1 &2 \end{array}
\right)
\end{equation*}
The restrictions of the weighted blowup $\mathcal F_i\to \mathbb{P}_i$ with centre at $P$ are divisorial contractions $f_1\colon \widetilde Y\to Y$ and $f_2\colon \widetilde Z \to Z$. These both have discrepancy $1$, and globalise the germs of weighted blowups with center at a $cE$ point classified in \cite[Theorem 3.2]{Mar96}. Further, $f_1$ initiates a Sarkisov link between $Y$ and a conic bundle, while $f_2$ initiates a link between $Z$ and an index $2$ Fano $3$-fold $Y_{4,6}\subset \mathbb{P}(1^2, 2^2, 3^2)$ which is the Fano variety number 40369 in \cite{GRDB}. 
\end{exa}


\bibliographystyle{amsalpha}

\bibliography{biblio}
\end{document}